\documentclass[10pt]{amsart}
\usepackage{amssymb}
\usepackage[mathscr]{euscript}
\usepackage[all]{xy}
\usepackage[dvips]{graphics}
\usepackage{amsmath,amssymb,epsfig}

\newtheorem{theorem}{Theorem}
\newtheorem{lemma}[theorem]{Lemma}

\newtheorem{proposition}[theorem]{Proposition}

\begin{document}
\title[Parity and Exotic Combinatorial Formulae]{Parity and Exotic Combinatorial  \\ Formulae for Finite-Type Invariants \\ of Virtual Knots}
\author{Micah Whitney Chrisman\and
Vassily Olegovich Manturov}

\begin{abstract} The present paper produces examples of Gauss diagram formulae for virtual knot invariants which have no analogue in the classical knot case. These combinatorial formulae contain additional information about how a subdiagram is embedded in a virtual knot diagram.  The additional information comes from the second author's recently discovered notion of parity. For a parity of flat virtual knots, the new combinatorial formulae are Kauffman finite-type invariants. However, many of the combinatorial formulae possess exotic properties.  It is shown that there exists an integer valued virtualization invariant combinatorial formula of order $n$ for every $n$ (i.e. it is stable under the map which changes the direction of one arrow but preserves the sign). Hence, it is not of Goussarov-Polyak-Viro finite-type.  Moreover, every homogeneous Polyak-Viro combinatorial formula admits a decomposition into an ``even" part and an ``odd'' part. For the Gaussian parity, neither part of the formula is of GPV finite-type when it is nonconstant on the set of classical knots. In addition, eleven new non-trivial combinatorial formulae of order 2 are presented which are not of GPV finite-type.
\newline
\newline
\emph{Keywords:} knot, virtual knot, parity, finite-type invariant, combinatorial formula
\end{abstract}
\maketitle
\section{Introduction}
Gauss diagram formulae for finite-type invariants of classical knots were first introduced by Polyak and Viro \cite{PV}. An important theorem of Goussarov \cite{GPV} states that every integer valued finite-type invariant of classical knots has a combinatorial formula. Those formulae for classical knot invariants which are also virtual knot invariants are entirely described by the Polyak algebra. However, not every finite-type invariant of virtual knots is given as an element of the Polyak algebra.

Indeed, there are two notions for finite-type invariants of virtual knots.  The one described by the Polyak algebra is called Goussarov-Polyak-Viro finite-type (GPV for short).  The other notion, due to Kauffman \cite{KaV}, is the natural generalization of Vassiliev invariant to the virtual case.  Every GPV finite-type invariant is of Kauffman finite-type.  On the other hand, not every Kauffman finite-type invariant is of GPV finite-type \cite{C2, C1, KaV}.  It follows that not every Kauffman finite-type invariant of virtual knots is representable by a combinatorial formula.

This observation leads to a natural question: Does there exist an extension of the Polyak algebra so that every integer valued Kauffman finite-type invariant of virtual knots is representable by a combinatorial formula? In the present paper, we construct a proper extension of the Polyak algebra and investigate its relation to the GPV and Kauffman finite-type invariants. The resulting extension contains non-trivial combinatorial formulae which are exotic in the sense that they are Gauss diagram formulae which are not of GPV finite-type. In particular, for every $n$, there is a combinatorial formula of order $n$ that is invariant under the virtualization move and hence, by \cite{C2}, is not of GPV finite-type.  

The method of extension is by the second named author's recent discovery of parity in knot theory \cite{INM}. A parity is an axiomatic way of labelling the arrows of a Gauss diagram that behaves nicely under the Reidemeister moves. The labels for a parity are elements of $\mathbb{Z}_2$, where each arrow is either formally \emph{even} or \emph{odd}.  Hence, the extension of the Polyak algebra given here is a formal labelling of arrows as even or odd, subject to some relations so as to make the resulting algebra invariant under the Reidemeister moves. The resulting quotient by these relations is denoted by $\mathscr{Q}_n$, where the natural number $n$ is the maximum number of arrows allowed in a diagram.

\begin{theorem} \label{thmone} Let $n \in \mathbb{N}$ and let $\mathscr{D}$ be the set of Gauss diagrams on $S^1$ (or $\mathbb{R}$).  For any parity $P$, there is a map $I_n[P]:\mathbb{Z}[\mathscr{D}] \to \mathscr{Q}_n$ such that if $v \in \text{Hom}_{\mathbb{Z}}(\mathscr{Q}_n,\mathbb{Q})$, then $v \circ I_n[P]$ is an invariant of virtual knots (resp., long virtual knots).  If $P$ is a parity of flat virtual knots (resp., long virtual knots), then $v \circ I_n[P]$ is a Kauffman finite-type invariant of order $\le n$.
\end{theorem}
\begin{proof} See Lemmas \ref{thmonelemm1} and \ref{thmonelemm2}.
\end{proof}

The exotic combinatorial formulae are constructed using several projections of $\mathscr{Q}_n$: the one which kills all diagrams with \emph{any} even arrows (denoted $\mathscr{T}_{n}$), the one which kills all diagrams with \emph{any} odd arrows (denoted $\mathscr{E}_n$), and the one which kills a diagram if it has fewer than $n$ arrows or \emph{all} even arrows (denoted $\mathscr{O}_{n}$). The main results may be summarized as follows.

\begin{theorem} \label{thm1} The group $\text{Hom}_{\mathbb{Z}}(\mathscr{T}_{n},\mathbb{Q})$ is a finitely generated free $\mathbb{Z}$-module of rank $\rho$ where $\rho$ is given by:
\[
\rho=\left\{\begin{array}{cl} n & \text{ for Gauss diagrams on } S^1 \\ n(n+3)/2 & \text{ for Gauss diagrams on } \mathbb{R}  
\end{array} \right. .
\]
If $P$ is the Gaussian parity, there is a generating set of combinatorial formulae which are of Kauffman finite-type but not of GPV finite-type. For any parity and $n \in \mathbb{N}$ there is a virtualization invariant combinatorial formula of order $n$ for every $n$ (and hence by \cite{C2}, not of GPV finite-type).
\end{theorem}  
\begin{proof} See Lemmas \ref{thm1lemm2}, \ref{thm1lemm3}, \ref{thm1lemm4} and \ref{thm1lemm1}. 
\end{proof}
Let $\mathscr{A}$ denote the set of Gauss diagrams where all arrows are drawn dashed and let $\mathscr{A}^{(1,0)}$ denote the dashed arrow diagrams where each arrow is labelled with either a $0$ or a $1$ arbitrarily. Recall from \cite{GPV} that there is a natural pairing $\left<\cdot,\cdot\right>:\mathbb{Z}[\mathscr{A}] \times \mathbb{Z}[\mathscr{A}] \to \mathbb{Z}$ defined on generators by $\left<D,E \right>=1$ if $D=E$ and $0$ otherwise. There is a similar pairing $\left<\left<\cdot,\cdot\right>\right>: \mathbb{Z}[\mathscr{A}^{(1,0)}]\times \mathbb{Z}[\mathscr{A}^{(1,0)}] \to \mathbb{Z}$ defined on generators by $\left<\left<D,E\right>\right>=1$ if $D=E$ as \emph{labelled} Gauss diagrams and $0$ otherwise. If $I:\mathbb{Z}[\mathscr{D}] \to \mathbb{Z}[\mathscr{A}]$ is the map which sums over all subdiagrams of a Gauss diagram and $F$ is GPV combinatorial formula, then $\left<F,I(\cdot)\right>$ is a long virtual knot invariant.  For such invariants, we have the following decomposition theorem.
\begin{theorem} \label{thm2} If $F$ is a homogeneous GPV formula of order $n$, then there is an $F^e \in \mathscr{E}_n$, called the even part of $F$, and $F^o \in \mathscr{O}_{n}$, called the odd part of $F$,  such that:
\[
\left<F,I(\cdot)\right>=\left<\left<F^e,I[P](\cdot) \right>\right>+\left<\left<F^o,I[P](\cdot) \right>\right>.
\]
For the Gaussian parity $P$, $F^e$ and $F^o$ are of Kauffman finite-type but not of GPV finite-type of order $\le m$ for any $m$, whenever $F$ is not constant on the set of classical knots.
\end{theorem}
\begin{proof} See Lemmas \ref{evenpol}, \ref{gpvpara}, \ref{evenodddecomp}, and \ref{thm2lemm1}.
\end{proof}
The proof of Theorem \ref{thm2} follows from a \emph{functorality} argument. However, many of the combinatorial formulae presented in the present paper have no analogue in the Polyak algebra.  The invariants from Theorem \ref{thm1} do not come from functorality.  In other words, they are not the even or odd part of any GPV combinatorial formula.  Moreover, it will be shown that there are Kauffman finite-type invariants of order two in $\mathscr{O}_2$ which cannot be written as a linear combination of the even and odd parts of any GPV formula.  

The organization of this paper is as follows.  In Section \ref{secbackground} we review the two notions of finite-type invariants, parity, and the methods which will be used in the proofs of the main results. In Section \ref{seccombpform}, we define the parity enhanced Poyak algebra and prove the various parts of Theorem \ref{thmone}.  In Section \ref{secoddarrow}, we prove the various parts of Theorem \ref{thm1}.  Section \ref{seceodecomp} contains a discussion of the Polyak algebra, functorality, and a proof of the parts of Theorem \ref{thm2}.  Finally, Section \ref{secappendix} contains a table which gives a generating set for combinatorial formulae in $\mathscr{O}_2$.  The formulae are linearly independent over $\mathbb{Z}$. 
\section{Background} \label{secbackground}
\subsection{Two Flavors of Finite-Type} In the present section, we review the two different notions of finite-type invariants for virtual knots. \subsubsection{Kauffman Finite-Type} In \cite{KaV}, Kauffman introduced the notion of graphical finite-type invariants.  This notion of finite-type invariant is the one which is most similar to the well-known diagrammatic formulation of finite-type invariants for classical knots.  In Kauffman's version, singular knots are replaced with \emph{4-valent graphs}. 

Let $K_{\bullet}:S^1 \to \mathbb{R}^2$ (or $\mathbb{R}^1 \to \mathbb{R}^2$) be a virtual knot \emph{diagram} (resp., long virtual knot diagram), where the transversal self-intersections are marked with one of three possible crossing types: over/under crossing, virtual crossing, or a graphical vertex.  In addition to planar isotopies, the Reidemeister moves, and virtual moves, one adds \emph{rigid vertex isotopy moves}\cite{KaV}. Any virtual knot invariant $v$ can be extended to an invariant of knotted 4-valent graphs by successive application of the following rule:
\[
v\left(\begin{array}{c} \scalebox{.18}{\psfig{figure=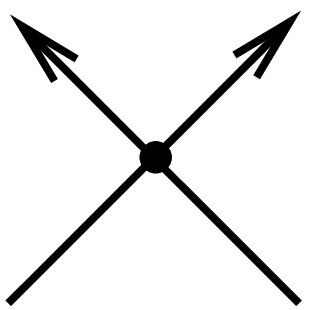}} \\ \end{array} \right):= v\left(\begin{array}{c} \scalebox{.18}{\psfig{figure=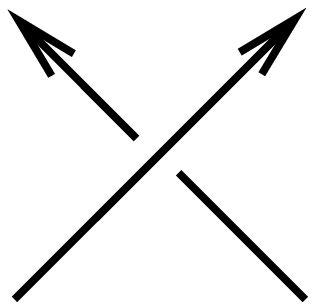}} \\ \end{array}\right)-v \left(\begin{array}{c} \scalebox{.18}{\psfig{figure=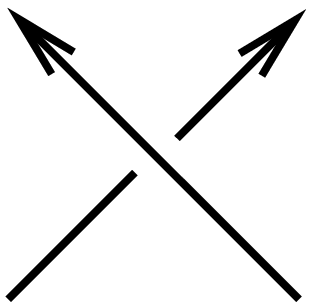}} \\\end{array}\right).
\]
A virtual knot invariant $v$ is said to be of Kauffman finite-type of order $\le n$ if $v(K_{\bullet})=0$ for all knotted 4-valent graphs $K_{\bullet}$ with more than $n$ graphical vertices.

Examples of Kauffman finite-type invariants arise in the same way as finite-type invariants for classical knots.  Let $f_K(A)$ denote the Kauffman $A$ polynomial of a virtual knot $K$.  The coefficient of $x^n$ in the power series expansion of $f_K(e^x)$ about $x=0$ is a Kauffman finite-type invariant (see \cite{KaV}). Other examples of Kauffman finite-type invariants can be found in \cite{Ma0}.

\subsubsection{Goussarov-Polyak-Viro Finite-Type} In addition to classical crossings, virtual crossings, and graphical vertices, knot diagrams may also have \emph{semi-virtual crossings}\cite{GPV}.  Like classical crossings, semi-virtual crossings may appear as ``over crossings'' or ``under crossings''.  Semi-virtual crossings are depicted as over or under classical crossings which are circumscribed by a small circle(see below).
\[
 \begin{array}{c} \scalebox{.2}{\psfig{figure=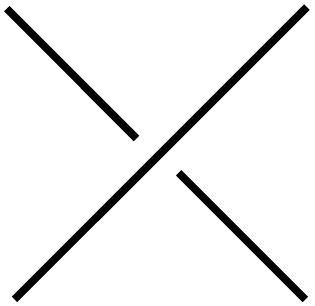}} \\ \text{Classical Crossing} \end{array} \hspace{.5cm} \begin{array}{c} \scalebox{.2}{\psfig{figure=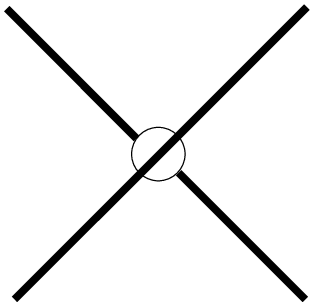}} \\ \text{Semi-virtual Crossing} \end{array} \hspace{.5cm} \begin{array}{c} \scalebox{.2}{\psfig{figure=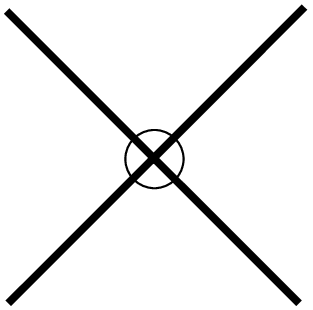}} \\\text{Virtual Crossing} \end{array} 
\]
Note that the definition of these crossing types depends on the orientation of the crossing.  The orientation of the strands has been omitted in the figure.

Virtual knot or long virtual knot invariants are extended to these semi-virtual diagrams by the following relation:
\[
v\left(\begin{array}{c} \scalebox{.18}{\psfig{figure=rightsemivirt.eps}} \\ \end{array}\right) := v\left(\begin{array}{c} \scalebox{.18}{\psfig{figure=rightcross.eps}} \\ \end{array}\right)-v\left(\begin{array}{c} \scalebox{.18}{\psfig{figure=virtcross.eps}} \\\end{array} \right).
\]
The extension to virtual knots with semi-virtual crossings is also denoted $v$.  A virtual knot invariant is said to be of \emph{Goussarov-Polyak-Viro finite-type} of order $\le n$ if $v(K_{\circ})=0$ for all semi-virtual knots $K_{\circ}$ with more than $n$ \emph{semi-virtual} crossings. For brevity, they are called GPV finite-type invariants. In \cite{GPV} it was shown that every GPV finite-type invariant is also of Kauffman finite-type.  However, the converse is not true (see Section \ref{virtmovesec}).

\subsubsection{Distinguishing between Kauffman and GPV finite-type invariants} \label{virtmovesec} There are two known methods for showing that a Kauffman finite-type invariant is not of GPV finite-type. The first is invariance under the virtualization move.  The second, more general method, is to employ twist lattices.  As both techniques are used in this paper, we discuss both of them at length.

The \emph{virtualization move} at a crossing in a knot diagram is as shown on the left hand side of Figure \ref{virtmv}. In \cite{KaV}, it was shown that the Jones polynomial is invariant under the virtualization move.  Hence, all of the the Kauffman finite-type invariants obtained from it also have this property.  More generally, Khovanov homology with arbitrary coefficients is also invariant under the virtualization move \cite{Kh}. However, in \cite{C2}, it was proved that there are no nonconstant integer valued GPV finite-type invariants that are invariant under the virtualization move.  Hence, to show that an integer valued Kauffman finite-type invariant is not of GPV finite-type, it is sufficient to prove that it is invariant under the virtualization move. For further discussion of the virtualization move, see \cite{FKM, Ma3}.
\begin{figure}[h]
\[
\begin{array}{c} \scalebox{.25}{\psfig{figure=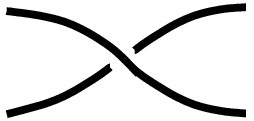}} \end{array} \leftrightharpoons \begin{array}{c} \scalebox{.25}{\psfig{figure=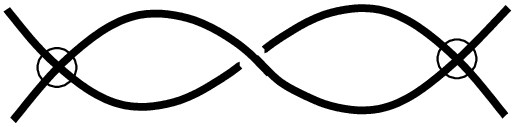}} \end{array},\,\,\,\,\begin{array}{c} \scalebox{.15}{\psfig{figure=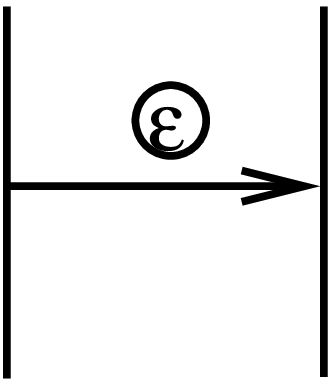}} \end{array} \leftrightharpoons \begin{array}{c} \scalebox{.15}{\psfig{figure=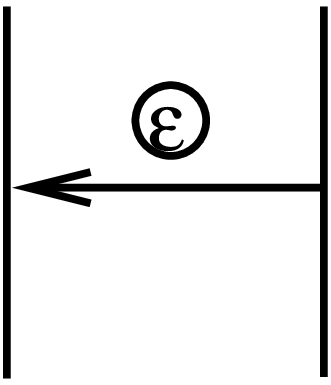}} \end{array}
\]
\caption{The Virtualization Move and Gauss Diagram Equivalent} \label{virtmv}

\end{figure}

Note that the virtualization move has a Gauss diagram description: it changes the direction of the arrow of the crossing without changing the sign (see right hand side of Figure \ref{virtmv}).

A more general method for distinguishing Kauffman and GPV finite-type invariants was given by the first named author in \cite{C1}. There, Eisermann's twist lattices \cite{E} were extended  to Kauffman and GPV finite-type invariants of virtual knots (also see \cite{E} for related references).  In the present paper, twist lattices will be used to show that some combinatorial formulae are of Kauffman finite-type but not GPV finite-type. 

Recall that $\mathscr{D}$ is the set of Gauss diagrams on $\mathbb{R}$ or $S^1$.  A \emph{twist sequence} is a function $\Phi:\mathbb{Z} \to \mathscr{D}$ such that the terms $\Phi(k)$ are identical outside a pair of disjoint intervals and inside the pair of intervals, the terms resemble one of the rows of the following diagram. 
\[
\begin{array}{|cccc|cccc|} \hline & \multicolumn{3}{c|}{\text{all arrows signed } \ominus} & \multicolumn{3}{c}{\text{all arrows signed } \oplus}& \\ \hline \ldots , & k=-2,  & k=-1, & k=0 & k=1, & k=2, & k=3,& \ldots \\ \hline \cdots & \scalebox{.12}{\psfig{figure=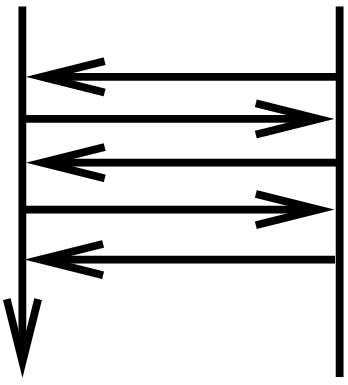}} &
\scalebox{.12}{\psfig{figure=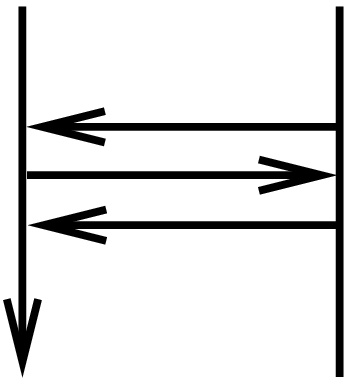}} &
\scalebox{.12}{\psfig{figure=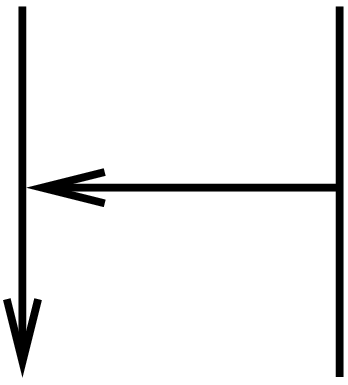}}&
\scalebox{.12}{\psfig{figure=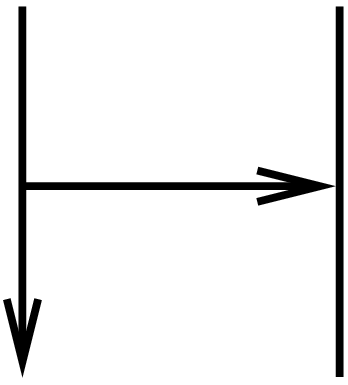}} &
\scalebox{.12}{\psfig{figure=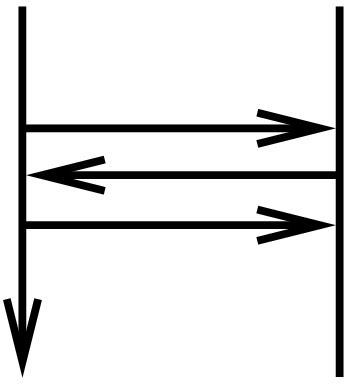}}&
\scalebox{.12}{\psfig{figure=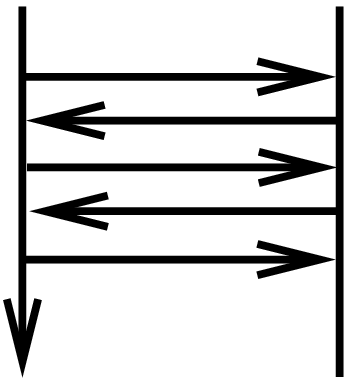}}& \cdots \\ \hline
\cdots & \scalebox{.12}{\psfig{figure=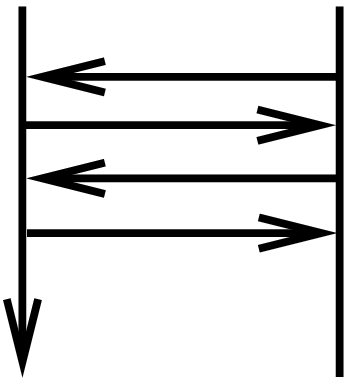}} &
\scalebox{.12}{\psfig{figure=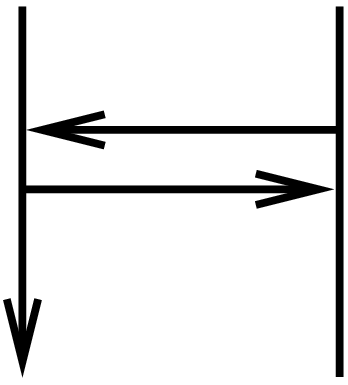}} &
\scalebox{.12}{\psfig{figure=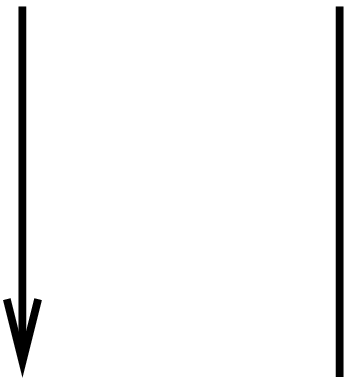}}&
\scalebox{.12}{\psfig{figure=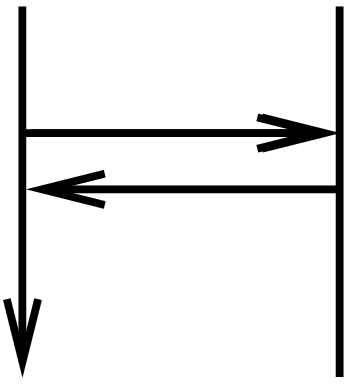}} &
\scalebox{.12}{\psfig{figure=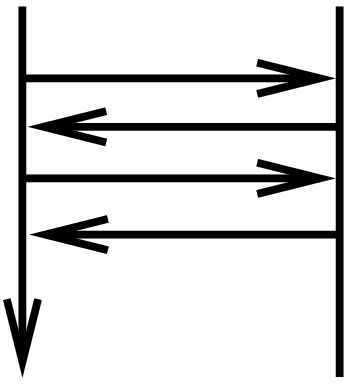}}&
\scalebox{.12}{\psfig{figure=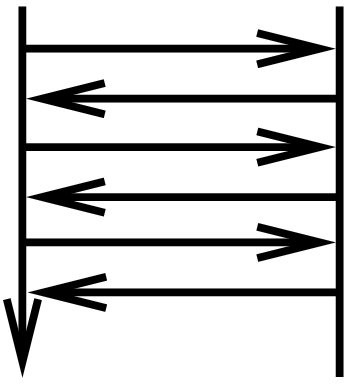}}& \cdots \\ \hline
\end{array}
\]
Note that the $k$-th term of the sequence in the first row has $|2k-1|$ arrows in the distinguished pair of arcs. Alternatively, we have the second row where every $k$-th term has $|2k|$ arrows in the distinguished pair of arcs.   

The \emph{type} of a twist sequence is given by the word $XYZ$. The character $X=E$ if there is an $E$ven number of arrows in each term and $X=O$ if there is an $O$dd number of arrows in each term. The character $Y=S$ if the ordering of the endpoints in the right vertical interval is the $S$ame as the ordering in the left vertical interval and $Y=B$ if the ordering on the right is $B$ackwards. The character $Z=L$ or $R$ according to whether the arrow(s) in the $k=1$ term point $L$eft or $R$ight. (Note: For twist sequences on $S^1$, the parameter $Z$ is dropped.) For more details on twist sequences and their types, the reader is referred to \cite{C1}.

A \emph{fractional twist sequence} is the same as twist sequence except that inside the pair of distinguished intervals, the terms are as in the following table.
\[
\begin{array}{|cccc|cccc|} \hline & \multicolumn{3}{c|}{\text{all arrows signed } \ominus} & \multicolumn{3}{c}{\text{all arrows signed } \oplus}& \\ \hline \ldots , & k=-2,  & k=-1, & k=0 & k=1, & k=2, & k=3,& \ldots \\ \hline
\cdots & \scalebox{.12}{\psfig{figure=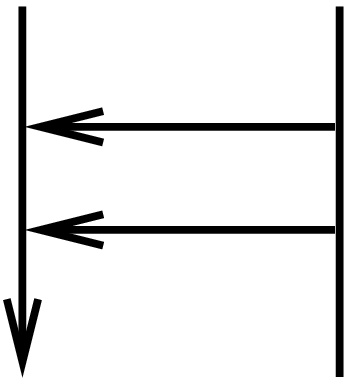}} &
\scalebox{.12}{\psfig{figure=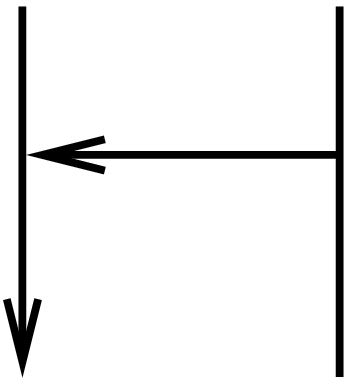}} &
\scalebox{.12}{\psfig{figure=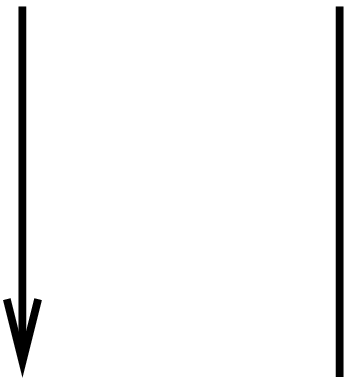}}&
\scalebox{.12}{\psfig{figure=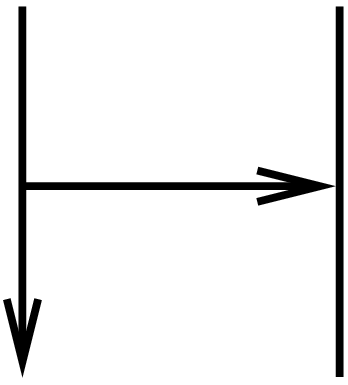}} &
\scalebox{.12}{\psfig{figure=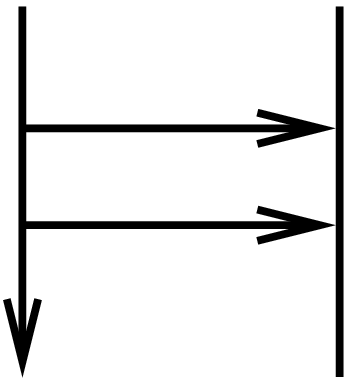}}&
\scalebox{.12}{\psfig{figure=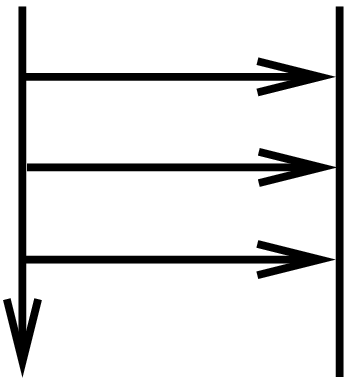}}& \cdots \\ \hline
\end{array}
\]
A fractional twist sequence is said to be of type $FYZ$ where $F$ stands for fractional and $Y$ and $Z$ are as above. Note that the distinguished arrows in a twist sequence of type $\text{FSZ}$ intersect one another.  On the other hand, the distinguished arrows in a twist sequence of type $\text{FBZ}$ do not intersect one another. 

A \emph{twist lattice} (\emph{fractional twist lattice}) is a function $\Phi:\mathbb{Z}^m \to \mathscr{D}$ such that each of the $m$ standard inclusions $\mathbb{Z} \to \mathbb{Z}^m \to \mathscr{D}$ is a twist sequence (resp. fractional twist sequence). The Gauss diagram $\Phi(\vec{0})$ is called the \emph{base} of the fractional twist lattice.  The following theorem is a generalization of a theorem of Eisermann \cite{E}.

\begin{theorem}[Chrisman \cite{C1}] \label{twist} A virtual knot or virtual long knot invariant $v:\mathbb{Z}[\mathscr{D}] \to \mathbb{Q}$ is of Kauffman finite-type (GPV finite-type) of order $\le n$ if and only if for every twist lattice (resp. fractional twist lattice) $\Phi:\mathbb{Z}^m \to \mathscr{D}$, the composition $v \circ \Phi:\mathbb{Z}^m \to \mathbb{Q}$ is a (discrete) polynomial of degree $\le n$.
\end{theorem}

In \cite{C1}, a twist lattice argument was used to show that the Kauffman finite-type invariants obtained from the Kauffman $A$ polynomial (i.e. the polynomial invariant which arises from the generalization of the Kauffman bracket to virtual knots\cite{KaV}) are not of GPV finite-type. This fact was originally observed for small orders in \cite{KaV}.  

\subsection{Parities} \label{pardef}
Let $D$ be a Gauss diagram.  Let $A(D)$ denote the set of arrows of $D$. If $D \leftrightharpoons D'$ is a Reidemeister move, then there is a one-to-one correspondence between arrows not involved in the move. For $w \in A(D)$, we denote the corresponding unaffected arrow as $w' \in A(D')$. Let $\mathscr{D}^{(1,0)}$ denote the set of Gauss diagrams where each arrow is labelled with an element of $\mathbb{Z}_2=\{0,1\}$.  A \emph{parity} is a function $P:\mathscr{D} \to \mathscr{D}^{(1,0)}$ satisfying the following four properties.
\begin{enumerate}
\item If $D \in \mathscr{D}$ has an arrow $x$ with consecutive endpoints then $P$ assigns the label $0$ to $x$.
\item If $D \in \mathscr{D}$ and $x,y \in A(D)$ have opposite sign and are embedded as the two affected arrows in a Reidemeister 2 move, then $P$ assigns the same label to $x$ and $y$. 
\item Suppose that $D \rightleftharpoons D'$ is a Reidemeister 3 move. Let $\{x,y,z\}$  denote the set of arrows of $D$ which are changed by the move and $\{x',y',z'\}$ the set of corresponding arrows in $D'$. Then $P$ assigns the label $1$ to either zero or two elements of $\{x,y,z\}$. If $t \in \{x,y,z\}$, then $P$ assigns the same label to $t$ and $t'$ in $\{x',y',z'\}$.
\item If $D \rightleftharpoons D'$ is any Reidemeister move, and $(y,y')$ is a corresponding pair of unaffected arrows, then $P$ assigns the same label to $y$ and $y'$.      
\end{enumerate}

The standard example of a parity is the \emph{Gaussian parity}.  Let $D$ be a Gauss diagram. To $D$ we associate its \emph{intersection graph}. Two arrows $a$ and $b$ are said to {\em intersect} (or to be {\em linked}) if their endpoints alternate on $\mathbb{R}$ or $S^1$. We write $(a,b)=(b,a)=1$ if $a$ and $b$ intersect and $(a,b)=0$ otherwise. The intersection graph is the graph with a vertex for each arrow of the diagram and an edge between two vertices $a$ and $b$ exactly when $(a,b)=1$ (see Figure \ref{intgraphdef}). 

\begin{figure}[h]
\[
\begin{array}{cc}
\scalebox{.25}{\psfig{figure=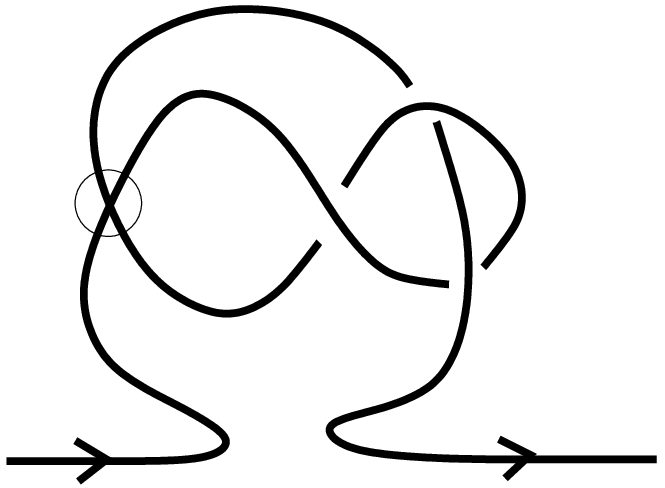}} & \hspace{1cm} \scalebox{.25}{\psfig{figure=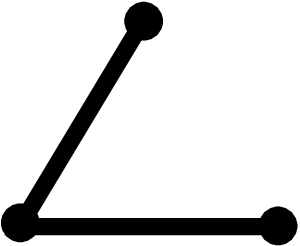}}
\end{array}
\]
\caption{A long virtual knot and its intersection graph.} \label{intgraphdef}
\end{figure}

Given a Gauss diagram $D$ and its intersection graph $G$, the Gaussian parity is defined as follows.  An arrow in $D$ is labelled $1$ if the degree of its vertex in $G$ is odd and a $0$ if the degree of its vertex in $G$ is even. It is easy to see that this definition satisfies the parity axioms. While quite simple in construction, the Gaussian parity has useful applications to realization and minimality problems in knot theory.  For example, the Gaussian parity was used in \cite{Ma01} to prove that there are nontrivial free knots.  This provided a first counterexample to a conjecture of Turaev (see also \cite{Gib}).

There are many examples of parities and their applications in knot theory. In \cite{Af}, parities were used to extend several knot polynomials.  In \cite{Ma7}, parities and justified parity were used to produce a sliceness obstruction for free knots. Additionally, there exists a parity of flat virtual knots which detects non-invertibility of free knots. More recently, the \emph{universal parity group} has been used to construct a natural map from knots in thickened surfaces to \emph{classical knots} \cite{Ma8}. Although parities are defined combinatorially, these observations imply that parities encode interesting topological and geometric properties of virtual knots.

\section{Theorem 1: Parity Enhanced Polyak Algebra}\label{seccombpform} \label{sec3}
\subsection{Construction of Parity Enhanced Polyak Algebra} The construction is the same for Gauss diagrams on $S^1$ and Gauss diagrams on $\mathbb{R}$. Consider $\mathscr{A}^{(1,0)}$, the collection of Gauss diagrams with \emph{dashed} arrows where each arrow has the label 1 or 0 (the empty Gauss diagram is by decree also in $\mathscr{A}^{(1,0)}$).  Note that the labels of this set are chosen arbitrarily.  Hence, a Gauss diagram with $n$ arrows has $2^n$ ways to attach labels to the arrows. 
\[
\underline{\text{Q1}}:\,\,\, \begin{array}{c} \scalebox{.35}{\psfig{figure=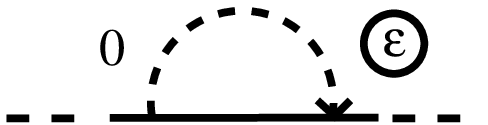}} \end{array}=0,\,\,\,\, \underline{\text{Q2}}:\,\,\, \begin{array}{c} \scalebox{.25}{\psfig{figure=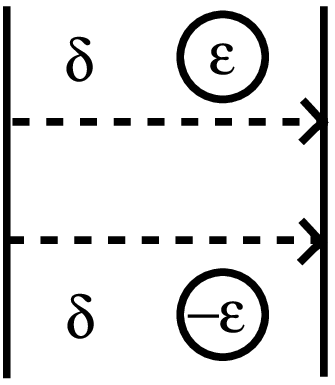}} \end{array}+\begin{array}{c} \scalebox{.25}{\psfig{figure=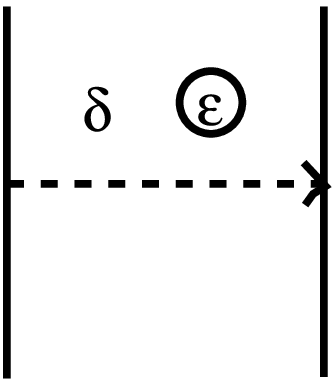}} \end{array}+\begin{array}{c} \scalebox{.25}{\psfig{figure=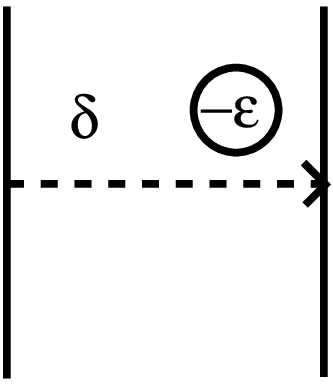}} \end{array}=0
\]
\begin{eqnarray*}
\underline{\text{Q3}}:\,\,\,  & & \begin{array}{c}\scalebox{.22}{\psfig{figure=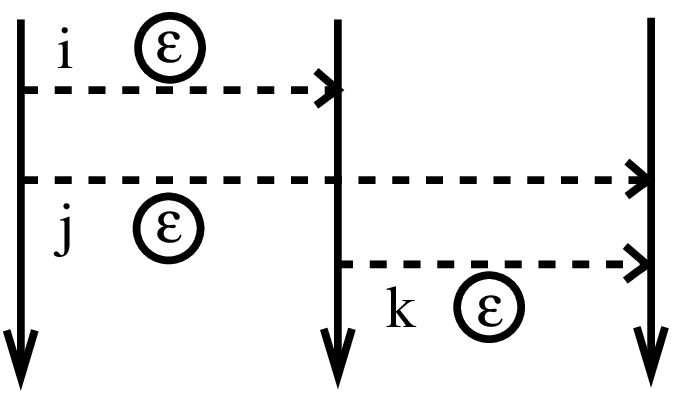}} \end{array}+\begin{array}{c}\scalebox{.22}{\psfig{figure=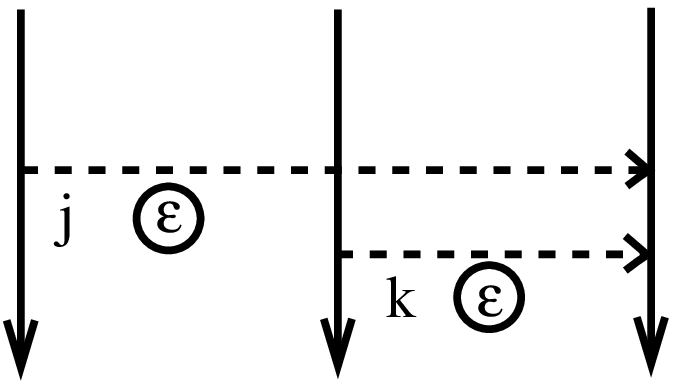}} \end{array}+\begin{array}{c}\scalebox{.22}{\psfig{figure=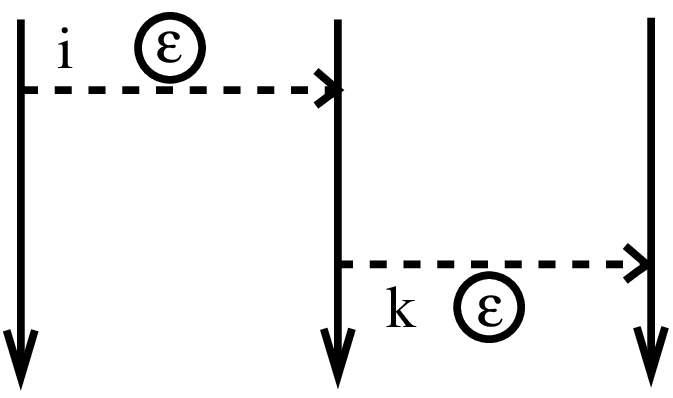}} \end{array}+\begin{array}{c}\scalebox{.22}{\psfig{figure=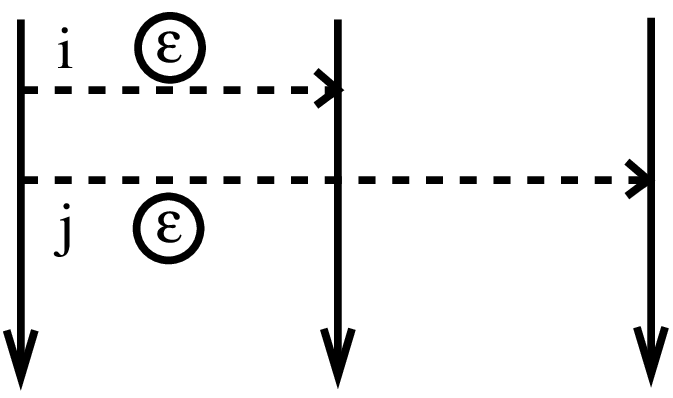}} \end{array} \\ & = & \begin{array}{c}\scalebox{.22}{\psfig{figure=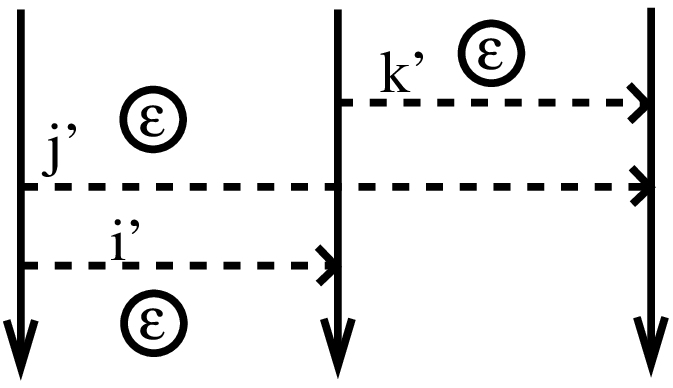}} \end{array}+\begin{array}{c}\scalebox{.22}{\psfig{figure=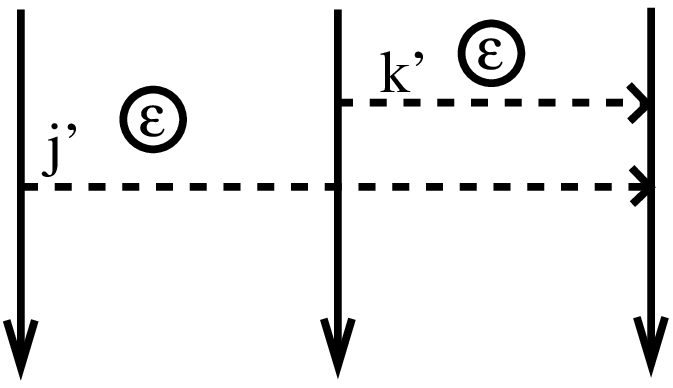}} \end{array}+\begin{array}{c}\scalebox{.22}{\psfig{figure=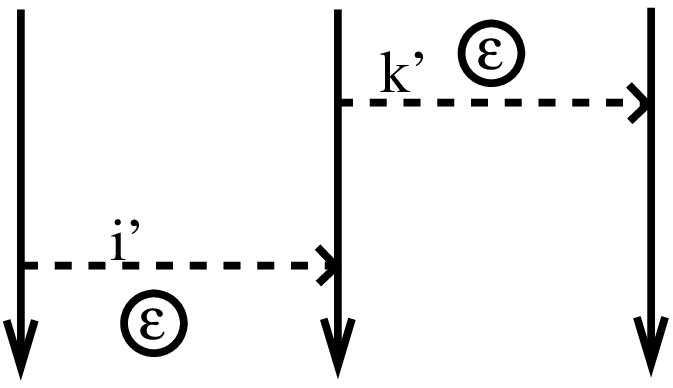}} \end{array}+\begin{array}{c}\scalebox{.22}{\psfig{figure=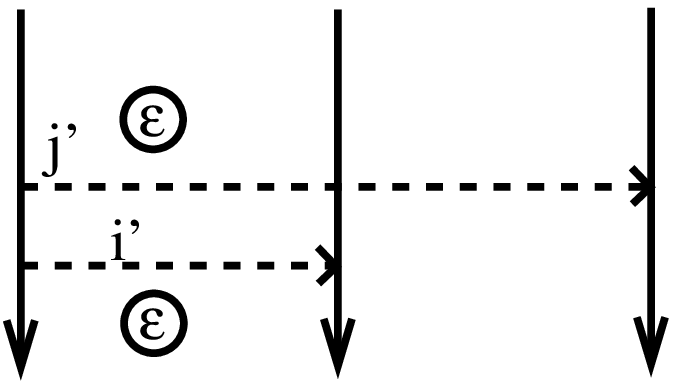}} \end{array} \\
\end{eqnarray*}
In Q3, we include all possibilities where $i=i'$, $j=j'$, $k=k'$, and $i+j+k\equiv i'+j'+k'\equiv 0 \pmod 2$. Note also that the $\delta$ in the Q2 relations refers to a label of 1 or 0.

Define $\Delta Q=\left< \text{Q1},\text{Q2},\text{Q3} \right>$.  Let $A_n \subset \mathscr{A}^{(1,0)}$ be the set of those diagrams having more than $n$ arrows. The \emph{parity enhanced Polyak algebra} is defined by the following quotients:
\[
\mathscr{Q}=\frac{\mathbb{Z}[\mathscr{A}^{(1,0)}]}{\Delta Q},\,\,\,\mathscr{Q}_n=\frac{\mathbb{Z}[\mathscr{A}^{(1,0)}]}{\left<\Delta Q,A_n\right>}.
\]

For Gauss diagrams on $\mathbb{R}$, the multiplication map for the algebra is defined in the same way as for the Polyak algebra: $\mu(D_1 \otimes D_2)$ is the concatenation $D_1 D_2$ of the diagrams.  The algebra structure of the parity enhanced Polyak algebra will not be used in the present paper.  Instead, we will refer to parity enhanced Polyak groups and use only its structure as a finitely generated abelian group.

Following \cite{GPV}, we define a map $i:\mathbb{Z}[\mathscr{D}^{(1,0)}] \to \mathbb{Z}[\mathscr{A}^{(1,0)}]$ on generators by $i(D)=D$ with all arrows drawn \emph{dashed} and all sign and $\mathbb{Z}_2$ markings unchanged.  Let $P$ be a parity.  Define $I[P]:\mathbb{Z}[\mathscr{D}] \to \mathbb{Z}[\mathscr{A}^{(1,0)}]$ on generators by:
\[
I[P](D)=\sum_{D' \subset P(D)} i(D').
\]
where the sum is taken over all subdiagrams $D'$ of $P(D)\in\mathscr{D}^{(1,0)}$, with the corresponding arrows in $D'$ labelled exactly as in $P(D)$. If $\pi_n:\mathscr{Q} \to \mathscr{Q}_n$ is the natural projection, we define $I_n[P]:\mathbb{Z}[\mathscr{D}] \to \mathscr{Q}_n$ by the composition $I_n[P]=\pi_n \circ I[P]$.

For any parity $P$, the linear functionals in $\text{Hom}_{\mathbb{Z}}(\mathscr{Q}_n, \mathbb{Q})$ give rise to virtual knot invariants. 

\begin{lemma}[proof of Theorem \ref{thmone}] \label{thmonelemm1} Let $P$ be any parity for Gauss diagrams on $S^1$ or $\mathbb{R}$. If $v \in \text{Hom}_{\mathbb{Z}}(\mathscr{Q}_n,\mathbb{Q})$, then $v \circ I_n[P]$ is a virtual knot invariant. 
\end{lemma}

\begin{proof} Given any Reidemeister relation, apply the map $I[P]$.  The distribution of zeros and ones in the $\Delta Q$ relations coincides with the parity definition in Section \ref{pardef}.  Collect terms in the image which are identical outside the affected arrows.  Each of the resulting groupings all lie in $\Delta Q$.  The fact that these all vanish in $\mathscr{Q}_n$ is sufficient to guarantee that $v \circ \pi_n \circ I[P]$ is an invariant.
\end{proof}

Combinatorial formulae my be defined in direct analogy to those in \cite{GPV}.  There is the pairing $\left<\left<\cdot, \cdot\right>\right>:\mathbb{Z}[\mathscr{A}^{(1,0)}] \times \mathbb{Z}[\mathscr{A}^{(1,0)}] \to \mathbb{Z}$ defined on generators by $\left<\left<D_1,D_2 \right>\right>=1$ if $D_1=D_2$ as labelled Gauss diagrams and $\left<D_1,D_2 \right>=0$ if $D_1 \ne D_2$.  A combinatorial formula of order $\le n$ is an element $F \in \mathbb{Z}[\mathscr{A}^{(1,0)}]$ such that all terms in $F$ have $\le n$ arrows and $\left<\left<F,r\right>\right>=0$ for all $r \in \Delta Q$.  For a parity $P$, a combinatorial formula $F$ defines a virtual knot or long virtual knot invariant $v_F$ by the rule:
\[
v_F(\cdot)=\left<\left<F,I[P](\cdot)\right>\right>.
\]
It follows from the definitions that $\left<\left<F,\cdot\right>\right> \in \text{Hom}_{\mathbb{Z}}(\mathscr{Q}_n,\mathbb{Q})$.  The collection of all such invariants for a parity $P$ are referred to as \emph{combinatorial $P$-formulae}.

\subsection{Example: Existence of Nontrivial Formulae} \label{nontrivsec} As an example of a combinatorial $P$-formula, consider the linear combination shown in Figure \ref{nontrivinv}. Note that only the terms with all $\oplus$ arrows are given. For each drawn term $D_{\oplus,\oplus}$, there are three additional terms $D_{\oplus, \ominus}$, $D_{\ominus, \oplus}$, $D_{\ominus, \ominus}$ which correspond to the three additional ways in which the arrows can be signed.  The coefficient of $D_{\epsilon_1,\epsilon_2}$ in the sum is the coefficient of $D_{\oplus,\oplus}$ times $\epsilon_1 \cdot \epsilon_2$ (this is the simplification convention of \cite{GPV}).
\begin{figure}[h]
\[
\begin{array}{c} \psfig{figure=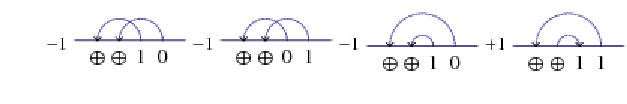} \end{array}
\]
\caption{A non-trivial formula $F_1$ of order $2$.} \label{nontrivinv}
\end{figure}

To show that this is a combinatorial $P$-formula, it is sufficient to check the relation $\text{Q1}$,$\text{Q2}$, and $\text{Q3}$. Since none of the diagrams has an arrow labelled 0 with consecutive endpoints, it follows that $\left<\left<F_1,\text{Q1} \right>\right>=0$.  

For a $\text{Q2}$ relation, note that all the affected arrows must have the same label.  As there is no pair of arrows having both the same direction and the same label and there are no diagrams having exactly one arrow, it follows that $\left<\left<F_1,\text{Q2} \right>\right>=0$ for all $\text{Q2}$ relations having two or fewer arrows. For $\text{Q2}$ relations having two or more arrows in each term, the simplification convention implies that $\left<\left<F_1,\text{Q2}\right>\right>=0$. 

Finally, the $\text{Q3}$ relations must be verified. Each term in a $\text{Q3}$ relation contains at least two arrows. However, the terms of $F_1$ have two arrows each.  Hence we must verify a number of \emph{six-term relations} i.e. the sum of terms of the $\text{Q3}$ relation which has only two arrows in every term. The three intervals of the $\text{Q3}$ relation can be embedded in six ways into $\mathbb{R}$.   For each of these cases, the labels of the arrows can be all zero or exactly one label can be zero.  If all the labels are zero, the relation is trivially satisfied.  Hence for each embedding of the intervals into $\mathbb{R}$, there are three cases to check. The case of one embedding is given in Figure \ref{relscheck}.  It is easily checked that $\left<\left<F_1,\text{LHS(Q3)}-\text{RHS(Q3)}\right>\right>=0$ for each of these cases.
\begin{figure}[h]
\begin{eqnarray*}
\scalebox{.15}{\psfig{figure=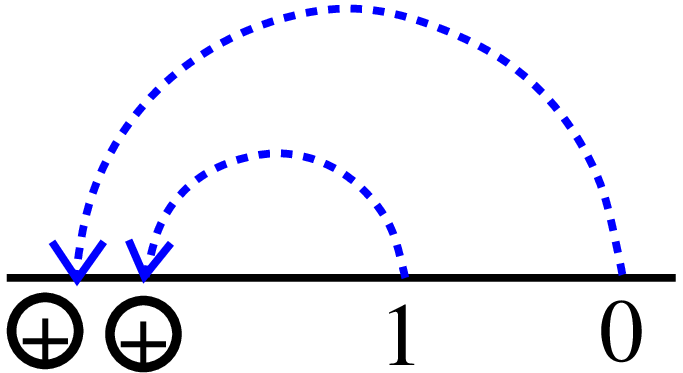}}\,\,+\,\,\scalebox{.15}{\psfig{figure=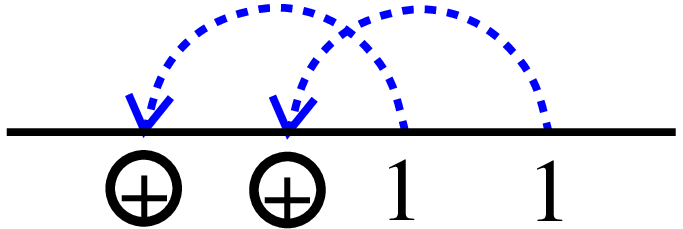}}\,\,+\,\,\scalebox{.15}{\psfig{figure=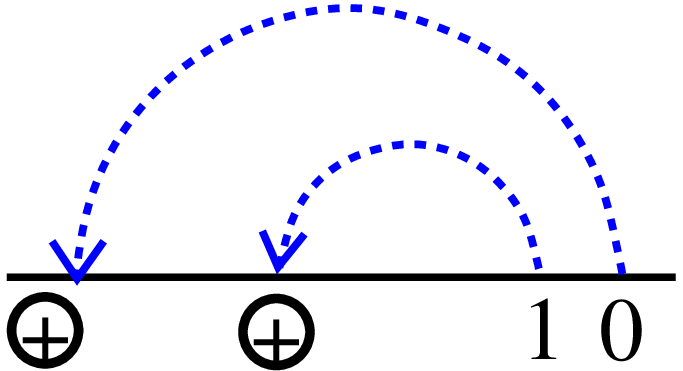}} &=& \scalebox{.15}{\psfig{figure=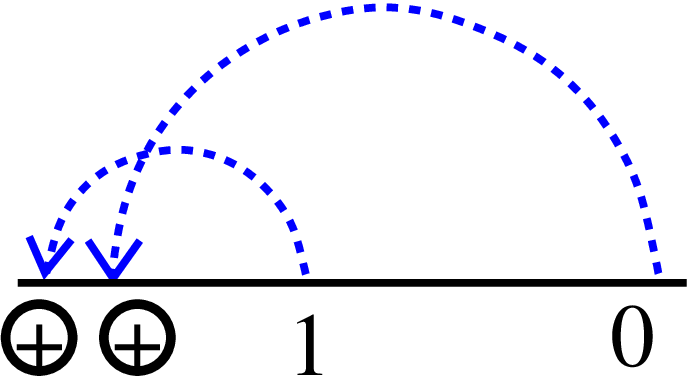}}\,\, + \,\, \scalebox{.15}{\psfig{figure=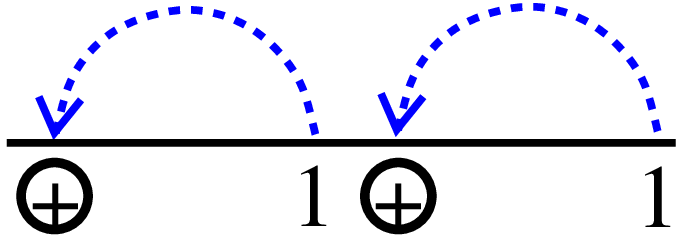}} \,\,+\,\,\scalebox{.15}{\psfig{figure=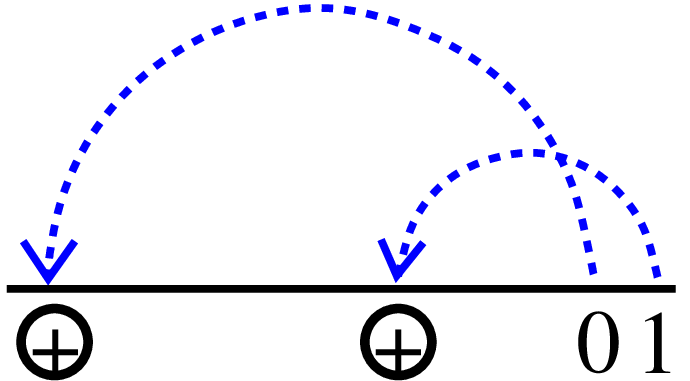}} \\
\scalebox{.15}{\psfig{figure=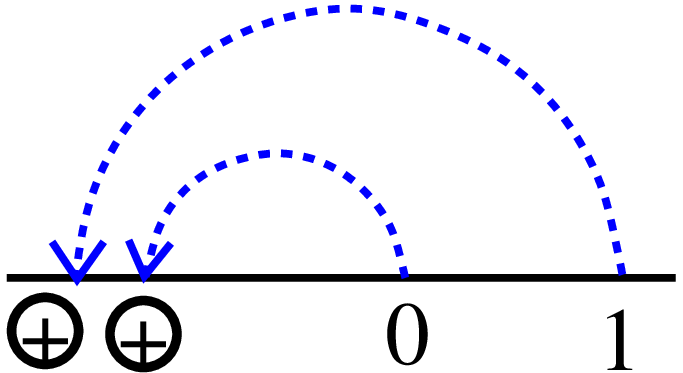}} \,\, + \,\, \scalebox{.15}{\psfig{figure=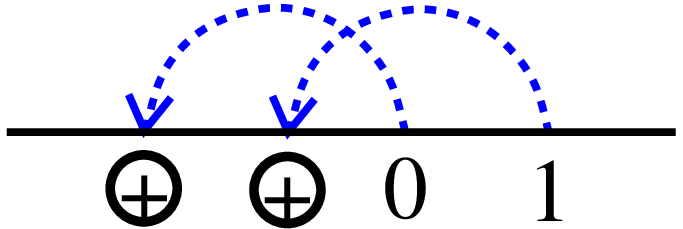}} \,\,+\,\,\scalebox{.15}{\psfig{figure=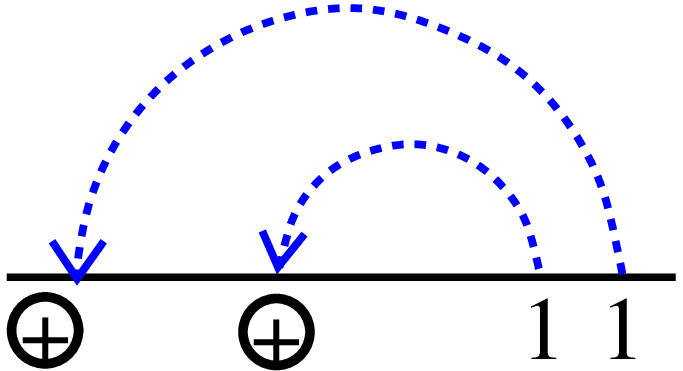}} &=& \scalebox{.15}{\psfig{figure=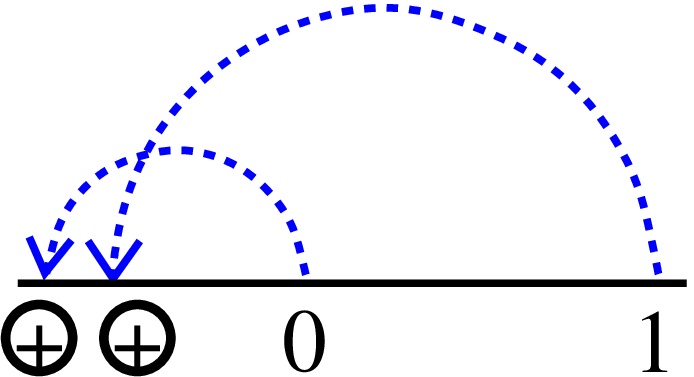}} \,\,+\,\,\scalebox{.15}{\psfig{figure=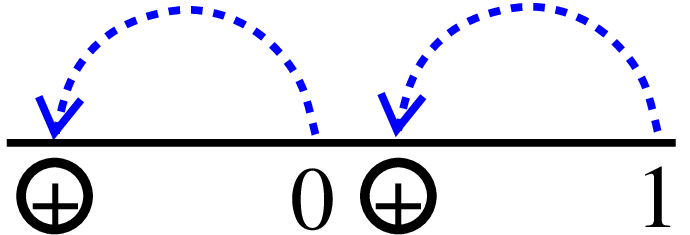}} \,\,+\,\, \scalebox{.15}{\psfig{figure=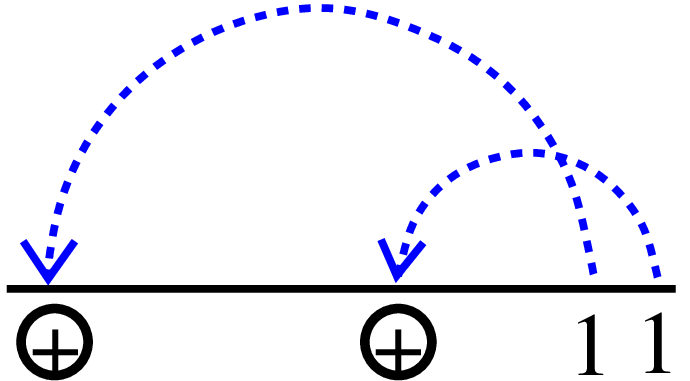}}\\
\scalebox{.15}{\psfig{figure=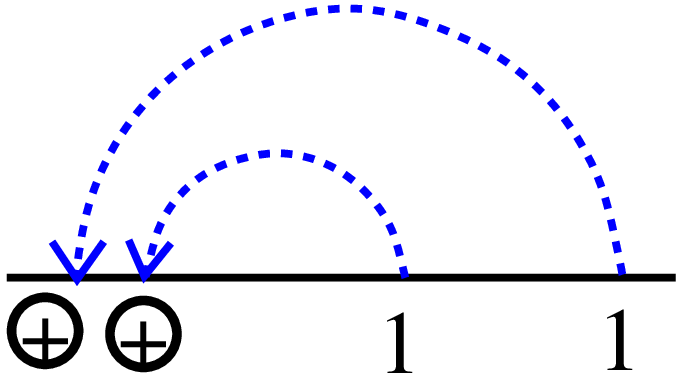}} \,\,+\,\, \scalebox{.15}{\psfig{figure=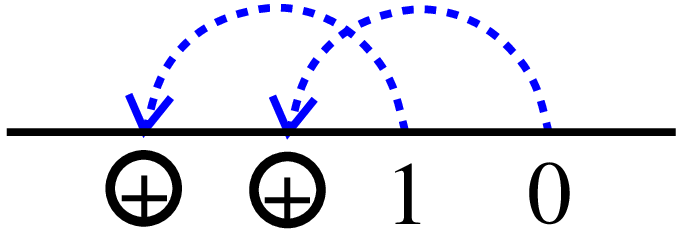}} \,\,+\,\, \scalebox{.15}{\psfig{figure=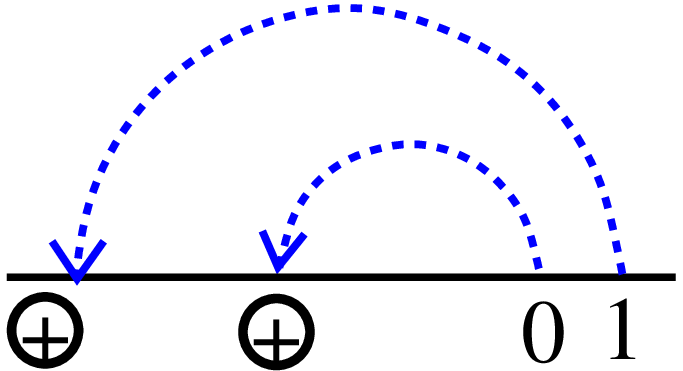}}  &=& \scalebox{.15}{\psfig{figure=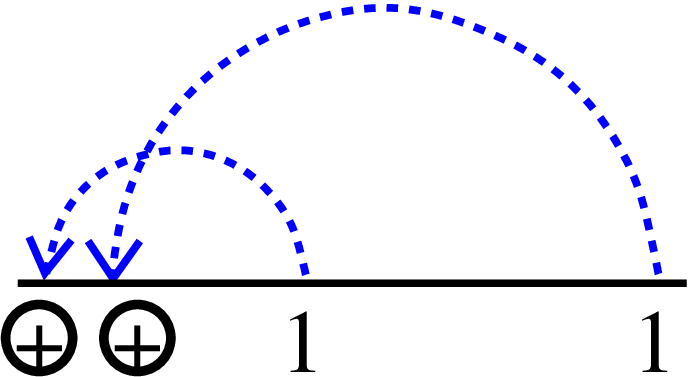}} \,\,+\,\, \scalebox{.15}{\psfig{figure=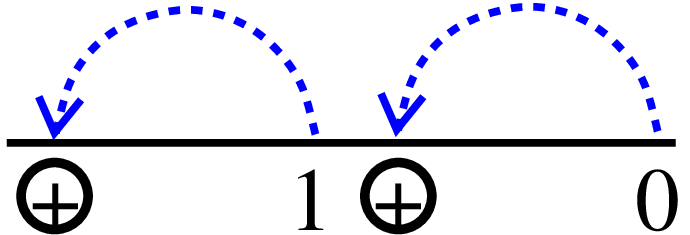}} \,\,+\,\, \scalebox{.15}{\psfig{figure=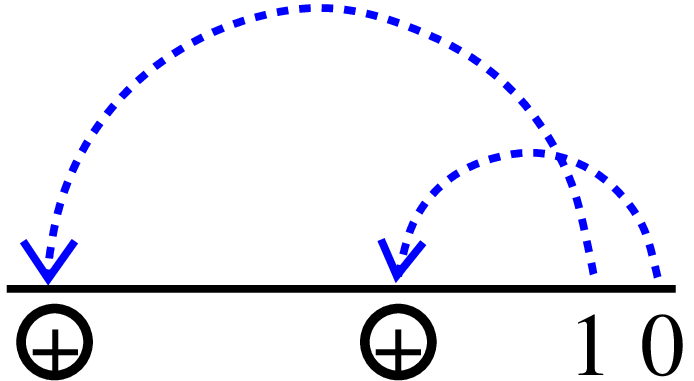}} 
\end{eqnarray*}
\caption{Three of 18 six-term relations needed to check invariance of $F_1$.} \label{relscheck}
\end{figure}

We note also that this combinatorial formula is independent of the GPV order two invariants. Recall from \cite{GPV} that the order $2$ GPV finite-type invariants for long knots are generated by the following formulae: 
\[
v_{21}(\cdot)=\left< \scalebox{.25}{\psfig{figure=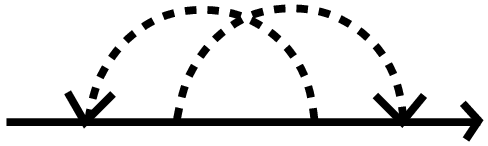}}, I(\cdot)\right>,\,\,\,\, v_{22}(\cdot)=\left< \scalebox{.25}{\psfig{figure=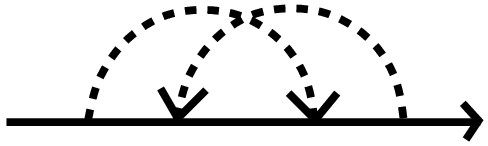}}, I(\cdot)\right> .
\]
where $I$ is the map which sums over all subdiagrams of a Gauss diagram (see Section \ref{functor}). Consider the Gauss diagram $D$ in Figure \ref{nontriv}. Let $P$ denote the Gaussian parity. We have $v_{21}(D)=v_{22}(D)=0$.  On the other hand, $\left<\left<F_1,I[P](D)\right>\right>=-2$.  We see that the invariant $v_{F_1}$ cannot be written as a linear combination of the invariants $v_{21}$ and $v_{22}$.
\begin{figure}[h]
\[
\begin{array}{cc}
\scalebox{.25}{\psfig{figure=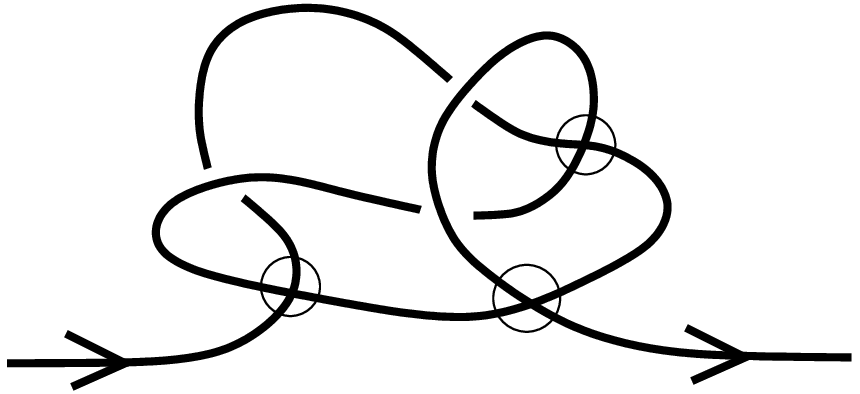}} & \scalebox{.25}{\psfig{figure=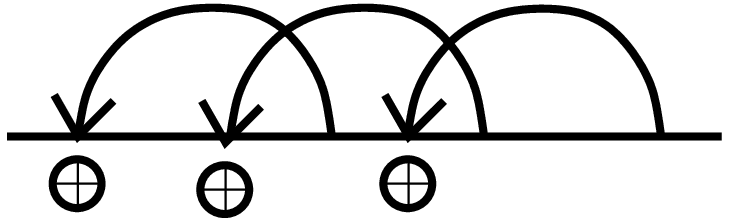}}\\
\end{array}
\]
\caption{$F_1$ is not a linear combination of $v_{21}$ and $v_{22}$.} \label{nontriv}
\end{figure}
\subsection{Parity of Flat Virtual Knots and Kauffman Finite-Type} Recall that a flat virtual knot \emph{diagram} is an equivalence class of virtual knot \emph{diagrams}. Two virtual knot diagrams are in the same flat equivalence class if they may be obtained from one another by a sequence of crossing changes, virtual moves, Reidemeister moves, and planar isotopies.

Flat virtual knots are represented in the plane as virtual knot diagrams where the over/under crossing information has been forgotten. The non-virtual crossings are indicated just as usual intersection points of two lines. For each such diagram, there corresponds a signed chord diagram. There is a virtual knot having either choice of sign in the flat equivalence class. However, there is only one choice of the direction of the arrow so that the chord diagram corresponds with the representation of the flat equivalence class. Changing both direction and sign of an arrow corresponds to switching the crossing from over to under or vice versa. For Gauss diagrams, this corresponds to changing the direction and sign of an arrow (see Figure \ref{gaussflat}).

\begin{figure}[h]
\[
D=\begin{array}{c}\scalebox{.25}{\psfig{figure=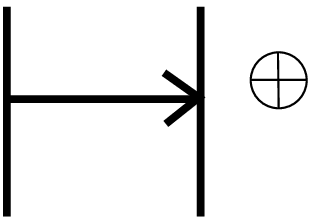}} \end{array},\,\,\,D'= \begin{array}{c}\scalebox{.25}{\psfig{figure=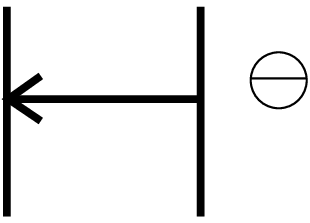}} \end{array}
\]
\caption{Equivalence Relation for Flat Virtual Knots} \label{gaussflat}
\end{figure}
This motivates the following definition. A parity $P$ is said to be a \emph{parity of flat virtual knots} if for all Gauss diagrams $D,D'$ which differ by the direction and sign of a single arrow, as in Figure \ref{gaussflat}, then $P$ assigns the same labels to the corresponding arrows of $D$ and $D'$. The Gaussian parity is an example of a parity of flat virtual knots. In addition, there are parities of flat virtual knots which arise from characteristic $\mathbb{Z}_2$-cohomology classes (see \cite{Ma6}).

\begin{lemma}[proof of Theorem 1] \label{thmonelemm2} If $P$ is a parity of flat virtual knots on $S^1$ or $\mathbb{R}$ and $v \in \text{Hom}_{\mathbb{Z}}\left(\mathscr{Q}_n,\mathbb{Q}\right)$, then $v \circ I_n[P]$ is a Kauffman finite-type invariant of order $\le n$.
\end{lemma}
\begin{proof} Suppose that $K$ is rigid 4-valent graph with one graphical vertex.  The resolution of this crossing may be expressed in terms of Gauss diagrams as $D-D'$, where $D$ and $D'$ are as in Figure \ref{gaussflat}. Now, every arrow of $D$ corresponds to an arrow of $D'$ having the same parity. We apply $I[P]$ to obtain:

\begin{eqnarray*}
I[P] : \begin{array}{c}\scalebox{.25}{\psfig{figure=switch1.eps}} \end{array} & \rightarrow & \begin{array}{c}\scalebox{.25}{\psfig{figure=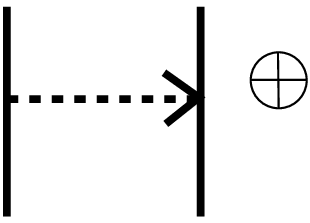}} \end{array}+\begin{array}{c}\scalebox{.25}{\psfig{figure=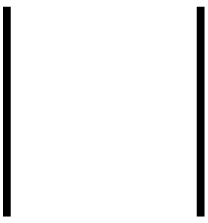}} \end{array} ,\\
I[P]: \begin{array}{c}\scalebox{.25}{\psfig{figure=switch2.eps}} \end{array} & \rightarrow & \begin{array}{c}\scalebox{.25}{\psfig{figure=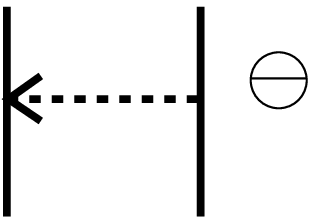}} \end{array}+\begin{array}{c}\scalebox{.25}{\psfig{figure=switch3.eps}} \end{array} .\\
\end{eqnarray*}
The second terms on the RHS of these equations are identical because $P$ is a parity of flat virtual knots.  Therefore, every term in $I[P](D-D')$ contains at least one arrow.  In general, $n+1$ graphical vertices will yield at least $n+1$ arrows in every nonzero term.  Hence, $v \circ I_{n}[P]$ vanishes on all 4-valent graphs with more than $n$ graphical vertices.
\end{proof}
\section{Theorem 2: Properties of Odd Arrow Groups}\label{secoddarrow} The odd arrow group $\mathscr{T}_n$ comes from the projection of $\mathscr{Q}_n$ which kills any diagrams which have an arrow labelled with a $0$. More specifically, let $E_1$ denote those diagrams having $\ge 1$ arrows labelled $0$ and perhaps some arrows labelled $1$.  We define the $n$-th odd arrow group, denoted $\mathscr{T}_n$, by the following projections. 
\[
\pi^{\mathscr{T}}_{n}:\mathscr{Q} \to \frac{\mathscr{Q}_n}{\left< E_1\right>}, \,\,\, \mathscr{T}_{n}:=\frac{\mathscr{Q}_n}{\left<E_1\right>}
\]
The symbol $\mathscr{T}$ is chosen due to the similarity between these invariants and discrete Taylor polynomials (see below).  In the following sections, we compute the rank of the groups $\text{Hom}_{\mathbb{Z}}(\mathscr{T}_n, \mathbb{Q})$. We also show that for any parity $P$, there is a virtualization invariant combinatorial formulae of order $n$ for every $n$. For the Gaussian parity, it is shown that there is a generating set of combinatorial $P$-formulae which are of Kauffman finite-type but not GPV finite-type.
\subsection{Simplifying the Projection to Odd Arrows} The first step in the proof of Theorem 2 is to find a presentation of the groups $\mathscr{T}_n$ in terms of only odd arrows. We consider each of the relations $\text{Q}1$, $\text{Q}2$ and $\text{Q}3$ in turn. 
\newline
\newline
\underline{Q1:} All Q1 relations are automatically satisfied since there is at least one arrow marked with the label 0.
\newline
\newline
\underline{Q2:} In this case, either both arrows are labelled 0 or both arrows are labelled 1. If both arrows are labelled 0, the relation is trivially satisfied.
\newline
\newline
If both arrows are labelled 1, the relation cannot be simplified further.  So the only Q2 relations needed are those with $\delta=1$ and all other arrows labelled 1.  Relations Q2 where all arrows (including those not in the drawn intervals) are labelled one will be denoted by $\text{Q2}^{(1)}$.
\newline
\newline
\underline{Q3:} Consider then the labellings of arrows involved in a Q3 relation. On the left hand side of the relation, there are arrows $\{i,j,k\}$.  On the right hand side, there are corresponding arrows $\{i',j',k'\}$.  By the definition of parity, either all three of the arrows are labelled 0 or only one is labelled 0.  If all arrows are 0, then the condition is trivially satisfied.

Otherwise, there are three cases: of the triple, only the arrows labeled $(i,i')$ are even or only the arrows labeled $(j,j')$ are even or only the arrows labeled $(k,k')$ are even. In each case, all but two of the labelled diagrams vanish.  The resulting three relations are given below.
\begin{eqnarray*}
\underline{\text{Q}3a^{(1)}:} &=& \begin{array}{c} \scalebox{.2}{\psfig{figure=q32.eps}}\end{array}=\begin{array}{c} \scalebox{.2}{\psfig{figure=q36.eps}}\end{array} \\
\underline{\text{Q}3b^{(1)}:} &=& \begin{array}{c} \scalebox{.2}{\psfig{figure=q33.eps}}\end{array}=\begin{array}{c} \scalebox{.2}{\psfig{figure=q37.eps}}\end{array} \\
\underline{\text{Q}3c^{(1)}:} &=& \begin{array}{c} \scalebox{.2}{\psfig{figure=q34.eps}}\end{array}=\begin{array}{c} \scalebox{.2}{\psfig{figure=q38.eps}}\end{array} \\
\end{eqnarray*}
We will refer to the collection $\{\text{Q}3a^{(1)},\text{Q}3b^{(1)},\text{Q}3c^{(1)}\}$ as the \emph{commutativity relations}. Note here that all of the arrows outside the intervals are also labelled 1. Let $\mathscr{A}^{(1)}$ denote those signed dashed arrow diagrams containing only arrows labelled $1$.  Let $A_n^{(1)}$ denote the set of dashed arrow diagrams having more than $n$ arrows, all of which are labelled 1. Collecting together the above statements gives the following lemma.
\begin{lemma} \label{on1thm} For every $n \ge 1$, there is an isomorphism of groups:
\[
\mathscr{T}_{n} \cong \frac{\mathbb{Z}[\mathscr{A}^{(1)}]}{\left<A_n^{(1)},\text{Q2}^{(1)},\text{Q}3a^{(1)},\text{Q}3b^{(1)},\text{Q}3c^{(1)} \right>}.
\]
\end{lemma}

\subsection{Some Example Invariants} The next step in the proof of Theorem 2 is to define a number of example invariants and show that they are realized by combinatorial $P$-formulae. The main argument relies upon discrete power series.  We review the necessary definitions for the discrete calculus in this section.

Consider first Gauss diagrams on $\mathbb{R}$.  Let $P$ be any parity of long virtual knots and $D \in \mathscr{D}$. Define $L_{\varepsilon}(D)$ ($R_{\varepsilon}$(D)) to be the number of arrows of $P(D)$ pointing left (respectively, right), signed $\varepsilon$, and labelled $1$. Define $\theta_L[P]:\mathscr{D} \to \mathbb{Z}$ and $\theta_R[P]:\mathscr{D} \to \mathbb{Z}$ as follows:
\begin{eqnarray*}
\theta_L[P](D) &=& L_{\oplus}(D)-L_{\ominus}(D), \\
\theta_R[P](D) &=& R_{\oplus}(D)-R_{\ominus}(D). \\
\end{eqnarray*}
Define $\theta_{L,R}^{n_1,n_2}[P]:\mathscr{D} \to \mathbb{Z}$ by the formula:
\[
\theta_{L,R}^{n_1,n_2}[P](D)=(\theta_{L}[P](D))^{n_1}(\theta_{R}[P](D))^{n_2}.
\]
\begin{proposition} $\theta_{L,R}^{n_1,n_2}[P]$ is an invariant of virtual knots for every parity $P$
\end{proposition}
\begin{proof} It is sufficient to show that $\theta_H[P]$, $H=L$ or $R$, satisfies the Reidemeister relations $\Omega 1$, $\Omega 2$, $\Omega 3$, where $H=L$ or $R$. Since the isolated arrow in an $\Omega 1$ move is always labelled $0$,  the number of arrows labelled $1$ on LHS and RHS of such a move is the same.  In an $\Omega 2$ move, both of the arrows have the same label and point in the same direction but have opposite sign.  The pair or arrows contribute $0$ to $\theta_H[P]$.

Consider finally the $\Omega 3$ move.  There are as many arrows labelled $1$ before and after the move.  Also, the move does not change the direction of the arrows.  The theorem of \"{O}stlund \cite{Ost} implies that it is sufficient to consider only the case where all the arrows involved in the move have the same sign. Thus, $\theta_H[P]$ has the same value on LHS and RHS of an $\Omega 3$ move. 
\end{proof}
\begin{lemma}[proof of Theorem \ref{thm1}] \label{thm1lemm2} If $P$ is the Gaussian parity, then $\theta_{L,R}^{n_1,n_2}[P]$ is not a $GPV$ finite-type invariant of order $n$ for any $n$.
\end{lemma}
\begin{proof} Consider a fractional twist sequence of type FSZ which is based on the empty Gauss diagram on $\mathbb{R}$.  The invariant $\theta_{L}^{n_1}[P]$ is nonzero on the even terms of the fractional twist sequence and zero on the odd terms of the twist sequence.  Since it is nonconstant and has infinitely many zeros, it cannot be a polynomial on this fractional twist sequence.  Hence, by Theorem \ref{twist}, $\theta_{L,R}^{n_1,n_2}[P]$ is not a $GPV$ finite-type invariant of order $n$ for any $n$
\end{proof}
It will now be shown that $\theta_{L,R}^{n_1,n_2}[P]$ can be represented by a combinatorial $P$-formula. The argument utilizes the discrete derivative.  Recall that if $F: \mathbb{Z} \to G$, where $G$ is an abelian group, the \emph{discrete derivative} at $x$ is given by:
\[
(\partial F)(x)=F(x+1)-F(x).
\]
The process can be iterated to find higher order discrete derivatives: $(\partial^n F)(x)=\partial (\partial^{n-1} F)(x)$. Also set $\partial^0$ to be the identity function.  If $F:\mathbb{Z}^m \to G$, partial derivatives may be defined by:
\[
(\partial_i F)(x_1,x_2,\ldots, x_i,\ldots x_n)=F(x_1,x_2,\ldots, x_i+1,\ldots x_n)-F(x_1,x_2,\ldots, x_i,\ldots x_n).
\]
It is easy to see that $\partial_i \partial_j=\partial_j \partial_i$ and hence the equality of mixed partials also holds in the discrete case.  

To define \emph{discrete power series}, a notion of power function is defined.  For $\alpha \in \mathbb{N}$, define:
\[
z^{\alpha}=z(z-1)(z-2)\cdot \ldots \cdot (z-\alpha+1).
\]
For $\alpha=0$, define $z^0=1$. If $\alpha=(\alpha_1,\ldots,\alpha_n)$ is an $n$-tuple of nonnegative integers, we define $z^{\alpha}=z_1^{\alpha_1} \cdot \ldots \cdot z_n^{\alpha_n}$ and $\partial^{\alpha}=\partial^{\alpha_1}_1 \cdots \partial^{\alpha_n}_n$  The degree of $z^{\alpha}$ is given by $|\alpha|=\alpha_1+\alpha_2+\ldots+\alpha_n$. The factorial of $\alpha$ is $\alpha!=\alpha_1! \alpha_2! \cdot \ldots \cdot \alpha_n !$. The discrete power series of $F$, $\Sigma_F:\mathbb{Z}^m \to G$ is defined to be:
\[
\Sigma_F(z_1,z_2,\ldots, z_m)=\lim_{n \to \infty} \sum_{|\alpha| \le n} \frac{(\partial^{\alpha} F)(\vec{0})}{\alpha !} z^{\alpha}.
\]
A straightforward application of the binomial theorem shows that $\Sigma_F=F$ (see for example \cite{C1}). This is the key observation for finding combinatorial formulae.  Let $\mathbb{Z}_{\le n}[z_1,z_2]$ denote the polynomials with integer coefficients in the variables $z_1$ and $z_2$ having total degree at most $n$.  Denote by $B_{\le n}$ those products of binomial coefficients as follows:
\[
B_{\le n}=\left\{\left( \begin{array}{c} n_1 \\ k_1 \end{array} \right)\left( \begin{array}{c} n_2 \\ k_2 \end{array} \right): k_1+k_2 \le n, 0 \le k_i \le n_i \right\}.
\]
\begin{lemma} \label{coefflemm} Every $h \in \mathbb{Z}_{\le n}[z_1,z_2]$ may be written as an integral linear combination of elements in $B_{\le n}$. In particular, the coefficient of $\left(\begin{array}{c} z_1 \\ k_1 \end{array} \right)\left( \begin{array}{c} z_2 \\ k_2 \end{array} \right)$ is given by:
\[
(\partial^{(k_1,k_2)} h)(\vec{0}).
\]  
\end{lemma}
\begin{proof} This follows from the discrete power series of $h$ and the fact that $z^{(k_1,k_2)}/k_1! k_2!$ can be identified with an element of $B_{\le n}$.
\end{proof}
\begin{lemma}[proof of Theorem \ref{thm1}] \label{thm1lemm3} For any parity $P$, the invariants $\theta_{L,R}^{n_1,n_2}[P]$ can be represented  by a parity enhanced combinatorial formula $F_{L,R}^{n_1,n_2}$, where $\left<\left<F^{n_1,n_2}_{L,R},\cdot\right>\right> \in \text{Hom}_{\mathbb{Z}}(\mathscr{T}_{n_1+n_2},\mathbb{Q})$.
\end{lemma}
\begin{proof} Consider the case of $\theta_{H}^n[P]$, where $H=L$ or $R$. By the binomial theorem, we may write $\theta_{H}^n[P](\cdot)$ as a polynomial of degree $n$ in the variables $z_1=H_{\oplus}(\cdot)$ and $z_2=H_{\ominus}(\cdot)$. By Lemma \ref{coefflemm}, this polynomial may be written as an integral linear combination of elements of $B_{\le n}$. Let $D(k_1,k_2)$ denote those dashed arrow diagrams having $k_1$ arrows of the form $H_{\oplus}$ and labelled $1$, and $k_2$ arrows of the form $H_{\ominus}$ and labelled $1$. We define a combinatorial formula $F=F_H^n$ for $\theta_{H}^n[P]$ as follows:
\[
F_H^n= \sum_{\stackrel{(k_1,k_2)}{0 \le k_1+k_2 \le n}} (\partial^{(k_1,k_2)} \theta_H^n[P])(\vec{0}) \sum_{D \in D(k_1,k_2)} D.
\] 
This is an integral linear combination of Gauss diagrams.  By definition of combinatorial formula, we have $\left<\left<F^n_H,I[P](X)\right>\right>=\theta^n_H[P](X)$ for all Gauss diagrams $X$.  

It must be shown that $F$ satisfies each of the relations for the parity enhanced Polyak algebra. Since all arrows of all terms in $F$ are labelled with a $1$, it follows that $\left<\left<F,\text{Q1} \right>\right>=0$. 

For a $\text{Q3}$ relation, we need only consider the relations $\text{Q3}a^{(1)}$, $\text{Q3}b^{(1)}$, and $\text{Q3}c^{(1)}$.  Since each of these relations preserves the number arrows pointing in a given direction and the signs of these arrows, it follows that $\left<\left<F,\text{LHS}(\text{Q3}d^{(1)})-\text{RHS}(\text{Q3}d^{(1)}) \right>\right>=0$ for $d=a,b,c$.

To complete the proof, it must be shown that the $\text{Q2}$ relations are satisfied.   From the $\text{Q2}$ relation and the definition of $F$, it follows that it is sufficient to show:
\[
(\partial^{(k_1,k_2)} \theta_H^n[P])(\vec{0})+(\partial^{(k_1-1,k_2)} \theta_H^n[P])(\vec{0})+(\partial^{(k_1,k_2-1)} \theta_H^n[P])(\vec{0})=0.
\]
This relation may be rewritten as follows:
\[
(\partial^{(k_1,k_2)}+\partial^{(k_1-1,k_2)}+\partial^{(k_1,k_2-1)})(z_1-z_2)^{n}(0,0) \\
= (\partial^{(1,1)}+\partial^{(0,1)}+\partial^{(1,0)})(\partial^{(k_1-1,k_2-1)})(z_1-z_2)^{n}(0,0) .
\]
Define $h(z_1,z_2)=(\partial^{(k_1-1,k_2-1)})(z_1-z_2)^{n}(z_1,z_2)$. Proceeding with the computation, we have:
\begin{eqnarray*}
(\partial^{(1,1)}+\partial^{(0,1)}+\partial^{(1,0)})(h)(z_1,z_2) &=&  h(z_1+1,z_2+1)-h(z_1,z_2+1)-h(z_1+1,z_2)+h(z_1,z_2)\\
&+& h(z_1,z_2+1)-h(z_1,z_2)+h(z_1+1,z_2)-h(z_1,z_2)\\
&=& h(z_1+1,z_2+1)-h(z_1,z_2)\\
&=& 0.
\end{eqnarray*}
The last equality follows from the definition of $h$.  Thus, $\left<\left<F,\text{Q2}\right>\right>=0$.  The general case of $\theta^{n_1,n_2}_{L,R}[P]$ follows similarly.  This completes the proof.
\end{proof}

Now that it has been shown that how to represent $\theta_{L,R}^{n_1,n_2}[P]$ as a combinatorial formula, we can show that $\text{Hom}_{\mathbb{Z}}(\mathscr{Q}_n, \mathbb{Q})$ has a virtualization invariant combinatorial formula for every $n$.  This phenomenon does not occur in the Polyak algebra (see \cite{C2}).  
\begin{lemma} [proof of Theorem \ref{thm1}] \label{thm1lemm4} For every $n \in \mathbb{N}$ and parity $P$, there exists a combinatorial $P$-formula of order $n$ which is invariant under the virtualization move.
\end{lemma}
\begin{proof} Note that the formula for $T=L_{\oplus}-L_{\ominus}+R_{\oplus}-R_{\ominus}$ is invariant under the virtualization move. Then $T^n$ is a virtualization invariant.  Moreover, $T^n$ can be written as an integral linear combination of the invariants $\theta_{L,R}^{j,k}[P]$, where $j+k=n$.  By Lemma 11, it follows that $T^n$ has a combinatorial formula.   
\end{proof}

For a Gauss diagram $D$ on $S^1$ and any parity $P$, let $N_{\varepsilon}(D)$ denote the number of arrows of $D$ labelled 1 and signed $\varepsilon$.  Let $\theta_N[P](D)=N_{\oplus}(D)-N_{\ominus}(D)$.  The above arguments show that for any $n \in \mathbb{N}$, $\theta_N^n[P]$ has a virtualization invariant combinatorial formula of order $n$.

\subsection{Rank of the Odd Arrow Groups} It remains to prove the claim of Theorem 2 that $\text{Hom}_{\mathbb{Z}}(\mathscr{T}_n,\mathbb{Q})$ is a finitely generated free abelian group and to compute its rank.  This is done using a short exact sequence argument and the example invariants of the previous section.

We will prove the rank theorem only in the case of diagrams on $\mathbb{R}$.  The case on the $S^1$ follows similarly.  Consider the short exact sequence of groups, where $\mathscr{T}_n \to \mathscr{T}_{n-1}$ is the natural surjection:
\[
\xymatrix{0 \ar[r] & \mathscr{K}_n \ar[r] & \mathscr{T}_{n} \ar[r] & \mathscr{T}_{n-1} \ar[r] & 0},
\]
\[
\mathscr{K}_n:=\frac{\left<A_{n-1}^{(1)},\text{Q2}^{(1)},\text{Q}3a^{(1)},\text{Q}3b^{(1)},\text{Q}3c^{(1)} \right>}{\left<A_n^{(1)},\text{Q2}^{(1)},\text{Q}3a^{(1)},\text{Q}3b^{(1)},\text{Q}3c^{(1)} \right>}.
\]
Taking the dual of the sequence gives another sequence:
\[
\xymatrix{0 \ar[r] & \text{Hom}_{\mathbb{Z}}(\mathscr{T}_{n-1},\mathbb{Q}) \ar[r] & \text{Hom}_{\mathbb{Z}}(\mathscr{T}_{n},\mathbb{Q}) \ar[r] & \text{Hom}_{\mathbb{Z}}(\mathscr{K}_{n}, \mathbb{Q}) }.
\]
\begin{lemma} There exists a finitely generated free abelian group $\mathscr{B}_n$ of rank $n+1$ such that the following sequence is exact.
\[
\xymatrix{0 \ar[r] & \text{Hom}_{\mathbb{Z}}(\mathscr{T}_{n-1},\mathbb{Q}) \ar[r] & \text{Hom}_{\mathbb{Z}}(\mathscr{T}_{n},\mathbb{Q}) \ar[r] & \text{Hom}_{\mathbb{Z}}(\mathscr{B}_{n}, \mathbb{Q}) \ar[r] & 0 }
\]
\end{lemma}
\begin{proof} Let $B_n$ denote the group of dashed arrow diagrams where all the arrows are labelled 1 and have no signs.  Let $\Delta |Q3|_n$ denote those commutativity relations having exactly $n$ unsigned arrows, all of which are labelled $1$. Define $\mathscr{B}_n=\frac{B_n}{\left<\Delta |Q3|_n \right>}$.  It is easy to see that there is a surjection $\mathscr{B}_n \to \mathscr{K}_n$ (see \cite{C2,Pol}).  Hence there is an injection $\text{Hom}_{\mathbb{Z}}(\mathscr{K}_n,\mathbb{Q}) \to \text{Hom}_{\mathbb{Z}}(\mathscr{B}_n, \mathbb{Q})$. 

Let $\beta_i$ denote the equivalence class of diagrams in $\mathscr{B}_n$ which have $i$ arrows pointing right and $n-i$ arrows pointing left. It is clear from the commutativity relations that $\mathscr{B}_n$ is a free abelian group which is generated by the $n+1$ elements  $\beta_0,\ldots, \beta_n$. Moreover, the rank of $\mathscr{B}_n$ is $n+1$.

To complete the proof, it remains to show that the composition below is a surjection:
\[
\xymatrix{\text{Hom}_{\mathbb{Z}}(\mathscr{T}_n,\mathbb{Q}) \ar[r] & \text{Hom}_{\mathbb{Z}}(\mathscr{K}_n,\mathbb{Q}) \ar[r] &  \text{Hom}_{\mathbb{Z}}(\mathscr{B}_n, \mathbb{Q})}.
\]
Consider the invariant $\theta^{n_1,n_2}_{L,R}[P]$ where $n_1+n_2=n$.  By Lemma \ref{thm1lemm3}, there is a combinatorial $P$-formula $F=F^{n_1,n_2}_{L,R}$ such that $\left<\left<F,I[P](\cdot)\right>\right>=\theta^{n_1,n_2}_{L,R}[P](\cdot)$. By following the various maps around, one can see that under the above composition, we have $\left<\left<F,\cdot\right>\right> \to n_1! n_2 ! \beta_{n_2}^*$, where $\beta_{n_2}^*$ is the functional which is $1$ on $\beta_{n_2}$ and $0$ every equivalence class distinct from $\beta_{n_2}$. Hence the composition is onto.  The claim regarding the kernel of the composition follows from the definitions of the various maps.
\end{proof}
\begin{lemma}[proof of Theorem \ref{thm1}] \label{thm1lemm1} For Gauss diagrams on $\mathbb{R}$, the rank of $\text{Hom}_{\mathbb{Z}}(\mathscr{T}_n, \mathbb{Q})$ is $n(n+3)/2$. The group is generated by the functionals $\left<\left<F^{n_1,n_2}_{L,R}, \cdot \right>\right>$ where $1 \le n_1+n_2 \le n$.
\end{lemma}
\begin{proof} Since $\text{Hom}_{\mathbb{Z}}(\mathscr{B}_n,\mathbb{Q})$ is free abelian and finitely generated, it follows that the short exact sequence above is split.  This leads to the direct sum decomposition:
\[
\text{Hom}_{\mathbb{Z}}(\mathscr{T}_n, \mathbb{Q}) \cong \text{Hom}_{\mathbb{Z}}(\mathscr{B}_n, \mathbb{Q}) \oplus \text{Hom}_{\mathbb{Z}}(\mathscr{T}_{n-1}, \mathbb{Q}). 
\]
The result follows by induction.  
\end{proof}
A similar argument can be used to show that the rank of $\text{Hom}_{\mathbb{Z}}(\mathscr{T}_n,\mathbb{Q})$ is $n$. This completes the proof of Theorem \ref{thm1}.
\section{Theorem 3: Decomposing GPV Formulae}\label{seceodecomp} In the present section, it is shown how to decompose a homogeneous GPV combinatorial formula into an even part and an odd part.  For the Gaussian parity, it is shown that the even and odd parts are of Kauffman finite-type but not GPV finite-type whenever the original formula is nonconstant on the classical knots. In addition we give an example to show that not all invariants of order $2$ in $\mathscr{O}_2$ arise from functorality.  These sections establish the proof of Theorem 3.
\subsection{Functorality, the Polyak Algebra, and Even Parts of GPV Formulae} \label{functor} The parity enhanced Polyak algebra is an extension of the notion of combinatorial formulae for finite-type invariants of classical knots. Recall the definition of the Polyak algebra from \cite{GPV}.  Let $\mathscr{A}$ be the set of dashed signed arrow diagrams.  For relations, we have P1, P2, and P3 which are simply Q1, Q2, and Q3 with all the labels erased.  Set $\Delta P=\left<\text{P1},\text{P2},\text{P3}\right>$ and, as before, $A_n$ the set of diagrams having more than $n$ arrows.  Then we define:
\[
\mathscr{P}=\frac{\mathbb{Z}[\mathscr{A}]}{\Delta P},\,\,\,\mathscr{P}_n=\frac{\mathbb{Z}[\mathscr{A}]}{\left<A_n,\Delta P \right>}, \,\, \pi_n^{\mathscr{P}}: \mathscr{P} \to \mathscr{P}_n.
\]
For $v \in \text{Hom}_{\mathbb{Z}}(\mathscr{P}_n,\mathbb{Q})$, the composition $v \circ \pi_n \circ I$ is a finite-type invariant of virtual knots, where:
\[
I(D)=\sum_{D' \subset D} i(D).
\]
The sum is taken over all subdiagrams of $D$. On the other hand, we have the \emph{even arrow groups} that arise from a projection of $\mathscr{Q}_n$. Let $O_1$ denote those diagrams in $\mathscr{A}^{(1,0)}$ which have the label $1$ at some arrow. Then define:
\[
\pi_n^{\mathscr{E}}:\mathscr{Q} \to \frac{\mathscr{Q}_n}{\left< O_1\right>}, \,\,\, \mathscr{E}_{n}:=\frac{\mathscr{Q}_n}{\left<O_1\right>}.
\]
We will show that certain combinatorial $P$-formulae are inherited from the Polyak groups.  The result follows essentially from \emph{functorality}. To be precise, let $P$ be any parity and let $f:\mathscr{D} \to \mathscr{D}$ denote the map which deletes all the arrows in a diagram which are odd relative to $P$.  The map $f$ is called the \emph{functorial map}.

\begin{lemma} Let $P$ be any parity and $f:\mathscr{D} \to \mathscr{D}$ be the functorial map.  If $D_1 \rightleftharpoons D_2$ is a Reidemeister move on Gauss diagrams, then either $f(D_1)=f(D_2)$ or $f(D_1) \rightleftharpoons f(D_2)$. Hence, if $K_1$ and $K_2$ are virtual knots having Gauss diagrams $D_1$ and $D_2$, respectively, and $K_1'$ and $K_2'$ are virtual knots having Gauss diagrams $f(D_1)$ and $f(D_2)$, then either $K_1'=K_2'$ or $K_1'$ and $K_2'$ are equivalent by a Reidemeister move and some detour moves.
\end{lemma}
\begin{proof} If the affected arrows in the move contain only even arrows, then $f(D_1) \rightleftharpoons f(D_2)$ as moves on diagrams.  If the affected arrows contain any odd arrows (so that the move is an $\Omega 2$ or $\Omega 3$ move), then $f(D_1)=f(D_2)$.
\end{proof}

The map $f$ indeed has the properties of a functor.  In the present case of combinatorial $P$-formulae, it is seen that $f$ behaves nicely with respect to the commutative diagram below. Indeed, let $\eta:\mathbb{Z}[\mathscr{A}] \to \mathbb{Z}[\mathscr{A}^{(1,0)}]$ be the map which labels every arrow of every diagram as 0. Let $\eta^e:\mathbb{Z}[\mathscr{A}^{(1,0)}] \to \mathbb{Z}[\mathscr{A}^{(1,0)}]$ be the map which projects any diagram with an arrow labelled one to the zero element of the group.  The above lemma shows that the following diagram commutes and descends to the quotient groups $\mathscr{P}$ and $\mathscr{Q}$.
\[
\xymatrix{
\mathbb{Z}[\mathscr{D}] \ar[rr]^{I[P]} \ar[d]_{f} & & \mathbb{Z}[\mathscr{A}^{(1,0)}] \ar[d]^{\eta^e} \\
\mathbb{Z}[\mathscr{D}] \ar[r]_{I} & \mathbb{Z}[\mathscr{A}] \ar[r]_{\eta} & \mathbb{Z}[\mathscr{A}^{(1,0)}]
}
\]
A consequence of this result is the following lemma which identifies the even arrow groups with the Polyak groups.
\begin{lemma}\label{evenpol} For every parity $P$, there is an isomorphism $\eta_n: \mathscr{P}_n \to \mathscr{E}_n$ for all $n \in \mathbb{N} \cup \{\infty\}$.  
\end{lemma}
For $F \in \mathbb{Z}[\mathscr{A}]$, define the \emph{even part} of $F$ to be $F^e=\eta(F)$. The previous lemma guarantees that if $F$ is a GPV combinatorial formula of order $\le n$, then $\left<\left<F^e,\cdot\right>\right> \in \text{Hom}_{\mathbb{Z}}(\mathscr{E}_n,\mathbb{Q})$. 

In the Gaussian parity, every arrow in the Gauss diagram $D$ of a classical knot is labelled 0. Hence it follows that $\left<F,I(D)\right>=\left<\left<F^e,I[P](D)\right>\right>$ for every GPV combinatorial formula $F$.  It is in this sense that the parity enhanced formulae extend the GPV formulae for finite-type invariants of classical knots.

\subsection{Odd parts of GPV formulae} The second part of the proof of Theorem 3 is to show that the complement of the even part of a homogeneous GPV formula is a combinatorial $P$-formula. We will define the odd part of a homogeneous formula and prove that it is in fact a combinatorial $P$-formula.  

An \emph{homogeneous} GPV combinatorial formula of order $n$ is a GPV combinatorial formula $F$ having exactly $n$ arrows in every summand.  Examples of order $2$ are the invariants $v_{21}$ and $v_{22}$. Also, there is the well-known Casson invariant. For order three, there exist examples which have been found by a \emph{Mathematica} program (see Figure \ref{hom3}).
\newline
\begin{figure}[h]
\centerline{\scalebox{1}{\psfig{figure=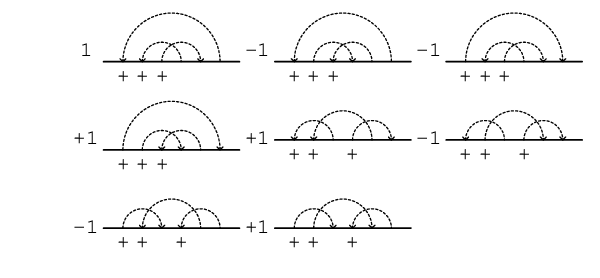}}}
\caption{A homogeneous GPV formula of order 3} \label{hom3}
\end{figure}

Recall that a GPV combinatorial formulae of order $\le n$ is an element $F \in \mathbb{Z}[\mathscr{A}]$ where every summand has $\le n$ arrows.  Hence $F$ may be written uniquely as $F=\sum_{D \in \mathscr{A}} \alpha_D D$ for some $\alpha_D \in \mathbb{Z}$, where all but finitely many $\alpha_D$ are zero.  For all $D \in \mathscr{A}$, set $\text{coeff}(D,F)=\alpha_D$. If $\displaystyle{F=\sum_{D \in \mathscr{A}} \text{coeff}(D,F) \cdot D}$ is a GPV combinatorial formula that is homogeneous of order $n$, define $F^o$ as follows.  If $\text{coeff}(D,F) \ne 0$, the contribution to $F^o$ is the sum over all labellings of the arrows of $D$ with zeros and ones so that the result has at least one arrow labelled 1. Hence there are a total of $2^{n}-1$ diagrams in $F^o$ arising from the contribution of $D$. This sum will be denoted by:
\[
\sum_{D^{(1)} \subset D} D^{(1)},
\]
where the sum is over all $2^n-1$ ways that the arrows can be labelled so that the result has at least one odd arrow.  Define
\[
F^o=\sum_{D \in \mathscr{A}}  \sum_{D^{(1)} \subset D} \text{coeff}(D,F) \cdot D^{(1)}.
\]
If $n=0$ and $F$ is the empty diagram, we define $F^o=0$.

The odd part of a GPV formula can be considered as a functional on a projection of $\mathscr{Q}_n$.  Indeed, let $E_n$ denote those subdiagrams which have $\ge n$ arrows labelled $0$.  Let $A^n$ denote those diagrams which have fewer than $n$ arrows. We define projections as follows:
\[
\pi_n^{\mathscr{O}}:\mathscr{Q}_n \to \frac{\mathscr{Q}_n}{\left<E_n, A^n \right>}, \,\, \mathscr{O}_n:=\frac{\mathscr{Q}_n}{\left<E_n, A^n \right>}.
\]
The distinct nonzero equivalence classes in $\mathscr{O}_n$ all have representatives having $n$ arrows, not all of which are labelled $0$.
\begin{lemma}[proof of Theorem \ref{thm2}] \label{gpvpara} If $F$ is a GPV combinatorial formula on $S^1$ or $\mathbb{R}$ which is homogeneous of order $n$, then:
\[
\left<\left< F^o, \cdot \right>\right> \in \text{Hom}_{\mathbb{Z}} \left( \mathscr{O}_n,
\mathbb{Q}\right).
\]
Hence, if $P$ is any parity, $\left<\left< F^o, I[P](\cdot)\right>\right>$ is an invariant of virtual knots (or virtual long knots).
\end{lemma}
\begin{proof} The formula vanishes on all diagrams having more than $n$ arrows, less than $n$ arrows, or $n$ even arrows.  Hence, it is sufficient to prove that $\left<\left<F^o,r\right>\right>=0$ for all $r \in \Delta Q$.
\newline
\newline
\underline{Q1:} Since $F$ is a GPV invariant. Any summand $D$ in $F$ having an isolated arrow must have $\text{coeff}(D,F)=0$.  By construction of $F^o$, it follows that $\left<\left<F^o, \text{Q1} \right>\right>=0$.
\newline
\newline
\underline{Q2:} We have that $\left<F,\cdot\right>$ satisfies all P2 relations.  Since $F$ is homogeneous of order $n$, the P2 relations may be divided into two types:
\[
\begin{array}{c} \scalebox{.2}{\psfig{figure=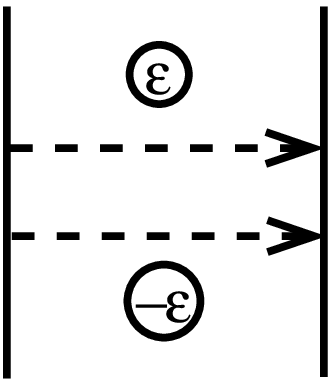}} \end{array}=0,\,\,\,\, \begin{array}{c} \scalebox{.2}{\psfig{figure=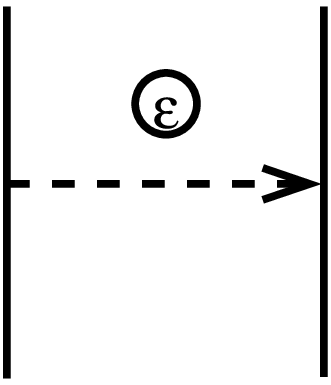}} \end{array}+\begin{array}{c} \scalebox{.2}{\psfig{figure=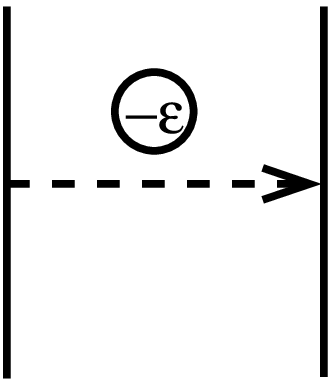}} \end{array}=0.
\]
The same division occurs in the Q2 relations for $F^o$. The two arrows involved in the $Q2$ relation are either both labelled 0 or both labelled 1. If they are both labelled 1, any diagram containing them has at least one arrow labelled 1. Hence, the value of $F^o$ on the two types agrees exactly with the value of $F$. Thus, $\left<\left<F^o,\text{Q2}\right>\right>=0$.

On the other hand, suppose that the arrows are both labelled 0. Then the diagram will be counted only when there is an arrow outside the drawn interval which is labelled 1. In this case, it occurs in all three diagrams in the P2 relation.  Hence the value of $F^o$ agrees with the value of $F$ on the two types and $\left<\left<F^o,\text{Q2}\right>\right>=0$.
\newline
\newline
\underline{Q3:} Since $F$ is homogeneous of order $n$, the P3 relations also split into two types:
\begin{eqnarray*}
\underline{A:} \,\,\, \begin{array}{c}\scalebox{.2}{\psfig{figure=q32.eps}} \end{array}+\begin{array}{c}\scalebox{.2}{\psfig{figure=q33.eps}} \end{array}+\begin{array}{c}\scalebox{.2}{\psfig{figure=q34.eps}} \end{array} &=&\\ \begin{array}{c}\scalebox{.2}{\psfig{figure=q36.eps}} \end{array}+\begin{array}{c}\scalebox{.2}{\psfig{figure=q37.eps}} \end{array}+\begin{array}{c}\scalebox{.2}{\psfig{figure=q38.eps}} \end{array} & &, \\
\end{eqnarray*}
\begin{eqnarray*}
\underline{B:}\,\,\, \begin{array}{c}\scalebox{.2}{\psfig{figure=q31.eps}} \end{array} &=&
\begin{array}{c}\scalebox{.2}{\psfig{figure=q35.eps}} \end{array}\,. \\
\end{eqnarray*}
The same division occurs for $F^o$ in Q3.

The sets $\{i,j,k\}$ and $\{i',j',k'\}$ contain either three arrows labelled 0 or two arrows labelled 1.

Suppose that all three are labelled 0.  Consider first the six term relation $A$.  The only way that any of the six diagrams is counted is if there is an arrow outside the drawn intervals which is labelled 1.  This arrow appears in all six diagrams, whence the value of $F^o$ on the relation agrees exactly with the value of $F$ on the relation. Hence, $\left<\left<F^o,\text{LHS}(\text{Q3})-\text{RHS}(\text{Q3})\right>\right>=0$. Similarly, $B$ is satisfied if all three drawn arrows are labelled 0.

Suppose that there are exactly two arrows in each of $\{i,j,k\}$ and $\{i',j',k'\}$ which are labelled 1. For the relation, the label on $x \in \{i,j,k\}$ agrees with the label on $x'\in \{i',i',k'\}$.  Moreover, any choice of two arrows from the sets $\{i,j,k\}$ and $\{i',j',k'\}$ contains at least one arrow which is labelled 1.  Therefore, each diagram in the relation $A$ has at least one odd arrow.  The diagrams in $B$ also have at least one arrow labelled 1.  Thus, $F^o$ agrees with $F$ on all such relations and $\left<\left<F^o,\text{LHS}(\text{Q3})-\text{RHS}(\text{Q3})\right>\right>=0$.
\end{proof}
\subsection{Even/Odd Decomposition Theorem} The main part of Theorem \ref{thm2} is the decomposition theorem. 

\begin{lemma}[proof of Theorem \ref{thm2}] \label{evenodddecomp} Let $F$ be a GPV combinatorial formula which is homogeneous of order $n$.  Then for any parity $P$, there is decomposition of $F$ into its even and odd parts:
\[
\left<F,I(\cdot) \right>=\left<\left<F^o,I[P](\cdot)\right>\right>+\left<\left< F^e, I[P](\cdot)\right>\right>.
\]
\end{lemma}
\begin{proof} Suppose that $n \ge 1$.  Let $D \in \mathscr{D}$ and $D' \subset P(D)$ be a subdiagram with corresponding arrows labelled as prescribed by the parity. If $D'$ does not have exactly $n$ arrows, then LHS and RHS of both equations are zero regardless of any parity considerations. Assume then that $D'$ has precisely $n$ arrows.

Suppose first that all the arrows of $D'$ are labelled zero. Then $\left<\left<F^o,i(D')\right>\right>=0$ and $\left<F,i(D')\right>=\left<\left<F^e,i(D')\right>\right>$.

On the other hand, suppose that at least one of the arrows in $D'$ is labelled one.  Then $\left<\left<F^e,i(D')\right>\right>=0$.  If $D'$ (forgetting all numerical markings) satisfies $\text{coeff}(D',F)=0$, then the desired decomposition is trivially true.  If $\text{coeff}(D',F) \ne 0$ then we have by definition of $F^o$ that:
\[
\left< F, i(D') \right>=\text{coeff}(D',F)=\left<\left<F^o,i(D') \right>\right>.
\]
Thus, the formula holds. If $n=0$, then $F$ is the empty diagram.  Since, $F^e$ is the empty diagram and $F^o=0$, the formula also holds in this case.
\end{proof}
\hspace{1cm}
\newline
\textbf{Example:} A natural question to ask is whether there are any invariants in $\mathscr{O}_{n}$ which do not arise from functorality. Indeed, there are combinatorial $P$-formulae whose associated invariant cannot be written as a linear combination of the even and odd parts of GPV formulae. Consider for example the formula $F_1$ in Figure \ref{nontrivinv} and let $P$ denote the Gaussian parity. Let $D$ denote the Gauss diagram in Figure \ref{nontriv}. As mentioned in Section \ref{nontrivsec}, we have that $\left<\left<F_1,I[P](D)\right>\right>=-2$.  On the other hand, if $F_{21}$ and $F_{22}$ denote the combinatorial formula for the generators $v_{21}$ and $v_{22}$ of the Polyak group $\mathscr{P}_2$, then:
\[
\left<\left<F_{21}^o,I[P](D) \right>\right>=\left<\left<F_{21}^e,I[P](D)\right> \right>=\left<\left<F_{22}^o,I[P](D) \right>\right>=\left<\left<F_{22}^o,I[P](D)\right>\right>=0.
\]
Hence, for the Gaussian parity, $F_1$ does not arise from functorality in the sense that is not a linear combination of the even and odds part of a $GPV$ formula.
\subsection{Even and Odd Parts are not of GPV Finite-Type} To complete the last part of Theorem 3, it is shown that in the Gaussian parity, the even and odd parts of a homogeneous GPV formula $F$ of order $n$ are not of GPV finite-type whenever $F$ is nonconstant on classical knots.  The main technique of the proof is twist lattices.

\begin{lemma} [proof of Theorem \ref{thm2}]\label{thm2lemm1} If $F$ is a homogeneous GPV combinatorial formula of degree $n>0$ which is nonconstant on the set of classical knots, and $P$ is the Gaussian parity, then its even part $F^e$ and its odd part $F^o$ are not of GPV finite-type of order $m$ for any $m$.
\end{lemma} 

\begin{proof} Let $v_F(\cdot)=\left<F,I(\cdot)\right>$. There is a classical knot on which $v_F$ is not zero.  Its Gauss diagram $D$ has all of its arrows even in the Gaussian parity.  

We construct a fractional twist lattice $\Phi_D:\mathbb{Z}^n \to \mathscr{D}$ on $D$ as follows. For each arrow, base a fractional twist sequence of type FBZ on the intervals which lie immediately to the right and left of the arrow tail. Then $\Phi_D(\vec{0})=D$ and if $z=(z_1,\ldots,z_n)$ is an $n$-tuple of odd integers, then every arrow of $\Phi_D(z)$ is odd in the Gaussian parity. Since $v_F$ is of GPV finite-type, $v_F\circ \Phi_D$ is a polynomial of order $\le n$. Suppose that $F^e$ generates a GPV finite-type invariant $u$ of order $\le t$.  Then $u \circ \Phi_D$ is a polynomial of order $\le t$.  Let $\delta: \mathbb{Z} \to \mathbb{Z}^n$ denote the diagonal embedding.  Since $u \circ \Phi_D$ is a polynomial of order $\le t$, $u \circ \Phi_D \circ \delta$ is a polynomial in a single variable of order $\le t$.  Since $u \circ \Phi_D \circ \delta(0)=v_F(D) \ne 0$ and $u \circ \Phi_D \circ \delta(1)=0$, $u \circ \Phi_D \circ \delta$ is not constant.  However, $u \circ \Phi_D \circ \delta (2z+1)=0$ for all $z \in \mathbb{Z}$.  As there is no such polynomial in a single variable, it follows that $F^e$ does not generate a GPV finite-type invariant of order $\le t$ for any $t$.

Now, since $F$ gives a polynomial on every fractional twist lattice, the even/odd decomposition theorem implies that either both $\left<\left<F^e, I[P](\cdot)\right>\right>$ and $\left<\left<F^o,I[P](\cdot)\right>\right>$ are polynomials or both are not polynomials.  By the preceding paragraphs, we see that there is always some twist lattice in which at least one of them is not a polynomial.  Thus, neither the invariant obtained from $F^e$ nor the invariant obtained from $F^o$ is of GPV finite-type.
\end{proof}

\section{Appendix}\label{secappendix}
Below is a generating set for the combinatorial $P$-formulae on $\mathscr{O}_{2}$ which are linearly independent over $\mathbb{Z}$, as found by the \emph{Mathematica} program \cite{CWR}. This program can be downloaded from the site listed in the reference \cite{CWR}. The program has been used to compute many higher order invariants as well, but we do not print them here.  The interested reader can compute them as well by following the instructions included with the program. 
\begin{enumerate}
\item \framebox{$\begin{array}{c} \psfig{figure=indycom22form1_better2.eps} \end{array}$}
\item \framebox{$\begin{array}{c} \psfig{figure=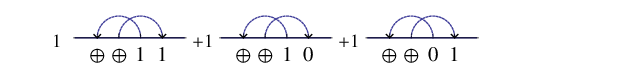} \end{array}$}
\item \framebox{$\begin{array}{c} \psfig{figure=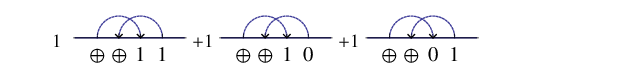} \end{array}$}
\item \framebox{$\begin{array}{c} \psfig{figure=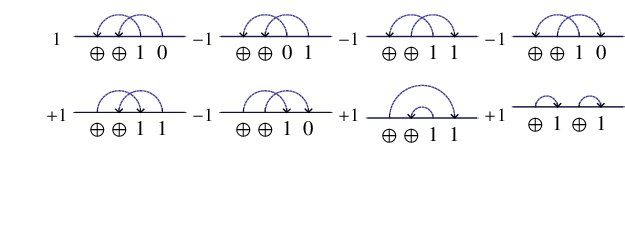} \end{array}$}
\item \framebox{$\begin{array}{c} \psfig{figure=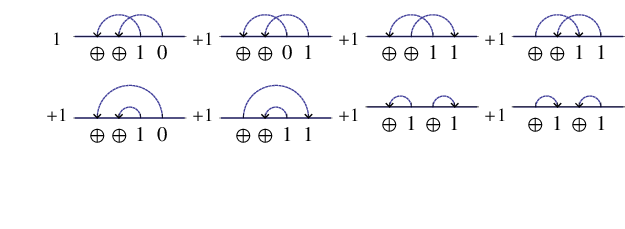} \end{array}$}
\item \framebox{$\begin{array}{c} \psfig{figure=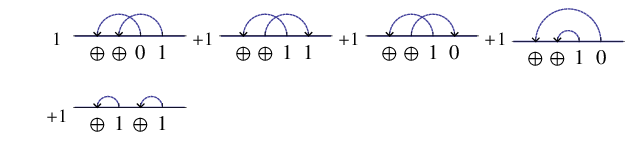} \end{array}$}
\item \framebox{$\begin{array}{c} \psfig{figure=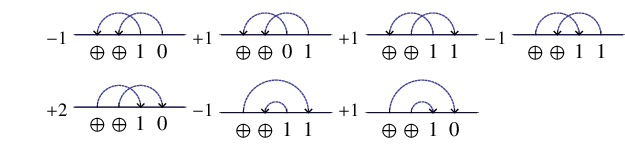} \end{array}$}
\item \framebox{$\begin{array}{c} \psfig{figure=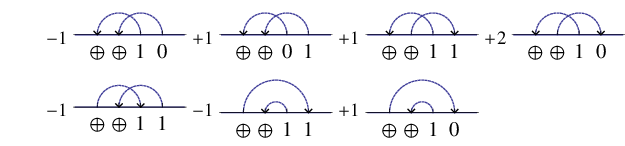} \end{array}$}
\item \framebox{$\begin{array}{c} \psfig{figure=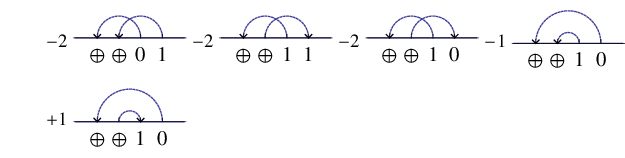} \end{array}$}
\item \framebox{$\begin{array}{c} \psfig{figure=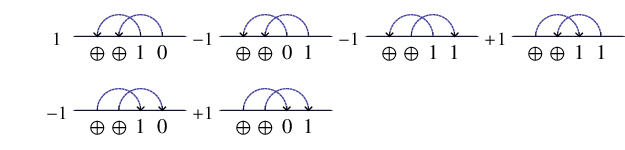} \end{array}$}
\item \framebox{$\begin{array}{c} \psfig{figure=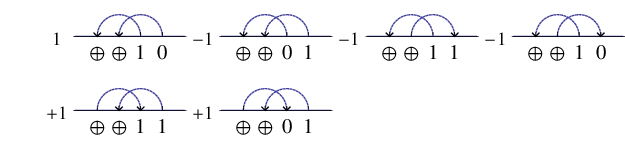} \end{array}$}
\end{enumerate}

Micah W. Chrisman:  Monmouth University, West Long Branch, NJ, USA\;

{\tt e-mail: mchrisma  at  monmouth dot edu}

Vassily Olegovich Manturov: People's Friendship University of
Russia, Faculty of Sciences, 3 Ordjonikidze St., Moscow, 117923

{\tt e-mail vomanturov at  yandex dot ru}

\bibliographystyle{plain}
\bibliography{bib_comm}

\begin{thebibliography}{10}

\bibitem{Af}
D.M. Afanasiev.
\newblock On amplification of virtual knot invariants by using parity.
\newblock {\em Sbornik Math.}, 201(6):785--800, 2010.

\bibitem{CWR}
Micah Chrisman.
\newblock \emph{Mathematica} program.
\newblock {\em http://www.monmouth.edu/$\sim$mchrisma/research.html}.

\bibitem{C2}
Micah Chrisman.
\newblock On the {G}oussarov-{P}olyak-{V}iro finite-type invariants and the
  virtualization move.
\newblock {\em Journal of Knot Theory and Its Ramifications}, 20(3):389--401,
  2011.

\bibitem{C1}
Micah~W. Chrisman.
\newblock Twist lattices and the {J}ones-{K}auffman polynomial for long virtual
  knots.
\newblock {\em J. Knot Theory Ramifications}, 19(5):655--675, 2010.

\bibitem{E}
Michael Eisermann.
\newblock A geometric characterization of {V}assiliev invariants.
\newblock {\em Trans. Amer. Math. Soc.}, 355(12):4825--4846 (electronic), 2003.

\bibitem{FKM}
Roger Fenn, Louis~H. Kauffman, and Vassily~O. Manturov.
\newblock Virtual knot theory---unsolved problems.
\newblock {\em Fund. Math.}, 188:293--323, 2005.

\bibitem{Gib}
Andrew Gibson.
\newblock Homotopy invariants of gauss words.
\newblock {\em ArXiv:Math.GT/0902.0062}.

\bibitem{GPV}
Mikhail Goussarov, Michael Polyak, and Oleg Viro.
\newblock Finite-type invariants of classical and virtual knots.
\newblock {\em Topology}, 39(5):1045--1068, 2000.

\bibitem{INM}
D.~Ilyutko, I.~Nikonov, and V.O. Manturov.
\newblock Parity in knot theory and graph-link theory.
\newblock {\em Journal of Mathematical Sciences}, to appear.

\bibitem{KaV}
Louis~H. Kauffman.
\newblock Virtual knot theory.
\newblock {\em European J. Combin.}, 20(7):663--690, 1999.

\bibitem{Ma3}
Vassily Manturov.
\newblock {\em Knot theory}.
\newblock Chapman \& Hall/CRC, Boca Raton, FL, 2004.

\bibitem{Ma6}
Vassily~O. Manturov.
\newblock Free knots and parity.
\newblock In {\em Proceedings of the Advanced Summer School on Knot Theory,
  Trieste, Series of Knots and Everything}.

\bibitem{Ma8}
Vassily~O. Manturov.
\newblock A functorial map from knots in thickened surfaces to classical knots
  and generalisations of parity.
\newblock {\em arXiv:1011.4640v1[math.GT]}.

\bibitem{Ma01}
Vassily~O. Manturov.
\newblock On free knots.
\newblock {\em ArXiv:Math.GT/0901.2214}.

\bibitem{Ma7}
Vassily~O. Manturov.
\newblock Parity and cobordisms of free knots.
\newblock {\em arXiv:math.GT/1001.2827}.

\bibitem{Ma0}
Vassily~O. Manturov.
\newblock Vassiliev invariants for virtual links, curves on surfaces and the
  {J}ones-{K}auffman polynomial.
\newblock {\em J. Knot Theory Ramifications}, 14(2):231--242, 2005.

\bibitem{Kh}
Vassily~O. Manturov.
\newblock Khovanov homology for virtual knots with arbitrary coefficients.
\newblock {\em Izvestiya: Mathematics}, 71(5):967--999, 2007.

\bibitem{Ost}
Olof-Petter \"{O}stlund.
\newblock Invariants of knot diagrams and relations among {R}eidemeister moves.
\newblock {\em J. Knot Theory Ramifications}, 10(8):1215--1227, 2001.

\bibitem{PV}
M.~Polyak and O~Viro.
\newblock Gauss diagram formulas for vassiliev invariants.
\newblock {\em International Mathematical Research Notes}, (11):445--453, 1994.

\bibitem{Pol}
Michael Polyak.
\newblock On the algebra of arrow diagrams.
\newblock {\em Lett. Math. Phys.}, 51(4):275--291, 2000.

\end{thebibliography}
\end{document}